\documentclass[12pt]{article}
\usepackage{amsmath,amssymb,amsthm}
\usepackage{color}
\usepackage[colorlinks=true]{hyperref}
\usepackage{graphicx}

\hypersetup{linkcolor=blue, urlcolor=green, citecolor=red}
\newtheorem{theorem}{Theorem}[section]
\newtheorem{lemma}[theorem]{Lemma}
\newtheorem{proposition}[theorem]{Proposition}
\newtheorem{definition}[theorem]{Definition}
\newtheorem{corollary}[theorem]{Corollary}

\newtheorem{remark}[theorem]{Remark}
\newtheorem{example}[theorem]{Example}
\numberwithin{equation}{section}

\title{Controllability of impulse controlled systems of heat equations coupled by constant matrices }

\author{Shulin Qin\thanks{School of Mathematics and Statistics, Wuhan University, Wuhan, 430072, China (shulinqin@yeah.net).}
\and
Gengsheng Wang\thanks{School of Mathematics and Statistics, Wuhan University, Wuhan, 430072, China (wanggs62@yeah.net).}
}

\begin{document}

 \date{ }
\maketitle

\begin{abstract}
This paper studies the approximate and null controllability for
 impulse controlled systems of heat equations coupled by a pair $(A,B)$ of constant matrices. We present a necessary and sufficient condition for the approximate controllability, which is exactly Kalman's controllability rank condition of $(A,B)$.
  We prove that when such a system is approximately controllable, the approximate controllability over an interval $[0,T]$ can be realized by adding controls at arbitrary $n$ different
  control instants $0<\tau_1<\tau_2<\cdots<\tau_n<T$, provided that $\tau_n-\tau_1<d_A$, where $d_A\triangleq\min\{\pi/|\mbox{Im} \lambda|\;:\; \lambda\in \sigma(A)\}$.
  We also show that in general,   such systems are not null controllable.

\end{abstract}

\noindent\textbf{Keywords} Impulse control,  approximate controllability, null controllability, systems of heat equations\\

\noindent\textbf{Mathematics Subject Classification (2010)} 93B05 35K40

\bigskip

\section{Introduction}

In this paper, we will study the null controllability and the approximate controllability for some impulse controlled systems of heat equations
coupled by constant matrices.
Impulse control belongs to a class of important controls and has wide
applications. In many cases impulse control is an interesting alternative to
deal with systems that cannot be acted on by means of continuous control
inputs, for instance, relevant control for acting on a population of
bacteria should be impulsive, so that the density of the bactericide may
change instantaneously; indeed continuous control would be enhance drug
resistance of bacteria (see \cite{TWZ} and \cite{Yang}).
Another application of impulse control in reality can be explained as follows:
In materials science, quenching is the rapid cooling of a workpiece to obtain certain material properties. A type of heat treating, quenching prevents undesired low-temperature processes, such as phase transformations, from occurring.  In ancient, a sequence of intermittent quenching is widely used in swordsmanship. We can regard such a quenching as an impulse control. Besides, there are many applications of impulse control theory to nanoelectronics (see Chapter 11 in \cite{Yang}).

To introduce  our controlled system, some notations are given in order.
Let $\Omega\subset \mathbb R^N$ (with $N\in\mathbb N^+\triangleq\{1,2,\dots\}$) be a bounded domain with a $C^2$ boundary $\partial\Omega$. Let $\omega\subset\Omega$ be an open and nonempty subset with its characteristic function  $\chi_\omega$. Write $\mathbb R^+ \triangleq (0, +\infty)$. Let $A$ and $B$ be respectively $n\times n$ and $n\times m$ (with $n,m\in \mathbb N^+$) real matrices, which are treated as
linear operators from $\mathbb R^{n}$ and $\mathbb R^{m}$ to $\mathbb R^{n}$  respectively.
Write $\mathbf{\Delta}\triangleq \mbox{diag}\{\Delta,\dots,\Delta\}$ (where there are $n$ Lapalacians). Define
 \begin{equation}\label{heat_1}
   \mathcal A\triangleq\mathbf{\Delta} -A
   \;\;\mbox{with}\;\;
   D(\mathcal A)\triangleq H^2(\Omega; \mathbb R^n) \cap H^1_0(\Omega; \mathbb R^n).
 \end{equation}
 (Namely, for each $\mathbf{z}=(z_1,\dots,z_n)^\top\in  D(\mathcal A)$, with $z_i\in H^2(\Omega; \mathbb R) \cap H^1_0(\Omega; \mathbb R)$, $i=1,\dots,n$, we define $ \mathcal A \mathbf{z}\triangleq \mathbf{\Delta}(z_1,\dots,z_n)^\top -A(z_1,\dots,z_n)^\top$.)
One can easily check that  $\mathcal A$ generates a $C_0$-semigroup $\{e^{\mathcal At}\}_{t\geq 0}$ over $L^2(\Omega; \mathbb R^n)$.
We treat $\chi_\omega$  as a linear and bounded operator on $L^2(\Omega; \mathbb R^n)$ in the following manner: For each $\mathbf{z}=(z_1,\dots,z_n)^\top\in L^2(\Omega; \mathbb R^n)$ (where $z_k\in L^2(\Omega; \mathbb R)$, $k=1,\dots,n$), we define that
$\chi_\omega \mathbf{z}\triangleq(\chi_\omega z_1,\dots,\chi_\omega z_n)^\top$.

 Consider the following  impulse controlled system of heat equations:
\begin{eqnarray}\label{heat-eq-ip}
\left\{\begin{array}{lll}
        \partial_t \mathbf{y}(t)- \mathcal A\mathbf{y}(t) = 0,  &  & t\in \mathbb R^+\setminus\{\tau_k\}_{k=1}^{p},\\
        \mathbf{y}(\tau_k) - \mathbf{y}(\tau_k-) = \chi_{\omega}B\mathbf{u}_k, &\;& k = 1,2,\ldots,p,\\
               \mathbf{y}(0)=\mathbf{y}_0\in L^2(\Omega; \mathbb R^n).
       \end{array}
\right.
\end{eqnarray}
Here,  $p \in \mathbb N^+$;  $0<\tau_1<\cdots<\tau_p<\infty$, which are called control instants;
 $\mathbf{u}_k=(u_{k1},\dots,u_{km})$, $k=1,\dots,p$, are  taken from $L^2(\Omega; \mathbb R^m)$
and  called impulse controls;  $\mathbf{y}(\tau_k-)$ denotes the left limit at $t=\tau_k$ for
the function $\mathbf{y}$. One can easily check that  the equation (\ref{heat-eq-ip}) is well-posed.
Write $\mathbf{y}(\cdot; \mathbf{y}_0,\{\tau_k\}_{k=1}^{p}, \{\mathbf{u}_k\}_{k=1}^{p} )$
   for the unique solution of (\ref{heat-eq-ip}). It is clear that
\begin{eqnarray}\label{solution-formula}
\mathbf{y}(t; \mathbf{y}_0, \{\tau_k\}_{k=1}^p, \{\mathbf{u}_k\}_{k=1}^p) = e^{\mathcal At}\mathbf{y}_0 + \sum_{1\leq k\leq p,\,\tau_k\leq t} e^{\mathcal A(t-\tau_k)}\chi_{\omega}B\mathbf{u}_k,
~ t\geq 0.
\end{eqnarray}

Throughout this paper, $\|\cdot\|$ and $\langle\cdot,\cdot\rangle$ denote the usual norm and inner product of $L^2(\Omega; \mathbb R^n)$, respectively; For each $C \in \mathbb R^{n\times n}$, we define
\begin{equation}\label{th-o-1-1}
  d_C \triangleq \min\Big\{ \pi / |\mbox{Im}\lambda| \;:\; \lambda \in \sigma(C)\Big\},
\end{equation}
where $\sigma(C)$ denotes the spectrum of $C$ and in the above definition we agree that $\frac{1}{0} = +\infty$. (We mention that $d_C$ takes value $+\infty$ in the case that $C$ has only real eigenvalues);
$A^*$, $B^*$ and $\mathcal A^*$ stand for the adjoint operators of $A$, $B$ and $\mathcal A$, respectively.

\begin{definition}\label{Def-2}
(i) Let $T>0$. The  system (\ref{heat-eq-ip}) is said to be  null controllable over $[0,T]$, if there is $p \in \mathbb N^+$, $\{\tau_k\}_{k=1}^{p} \subset (0, T)$ (with $\tau_1<\cdots<\tau_p$) so that
 for each  $\mathbf{y}_0 \in L^2(\Omega; \mathbb R^n)$, there is $\{\mathbf{u}_k\}_{k=1}^{p} \subset L^2(\Omega; \mathbb R^m)$   satisfying that
\begin{equation}\label{Def-2-1}
  \mathbf{y}(T; \mathbf{y}_0, \{\tau_k\}_{k=1}^{p}, \{\mathbf{u}_k\}_{k=1}^{p}) = 0  \;\;\mbox{in}\;\;  L^2(\Omega; \mathbb R^n).
\end{equation}
(ii) The  system (\ref{heat-eq-ip}) is said to be null controllable, if for each $T>0$, it is null controllable over $[0,T]$.
\end{definition}

 \begin{definition}\label{Def-3}
   (i) Let $T>0$. The  system (\ref{heat-eq-ip}) is said to be approximately controllable over $[0,T]$, if there is $p \in \mathbb N^+$, $\{\tau_k\}_{k=1}^{p} \subset (0, T)$ (with $\tau_1<\cdots<\tau_p$) so that for any $ \varepsilon>0$ and any
   $\mathbf{y}_0, \mathbf{y}_1 \in L^2(\Omega; \mathbb R^n)$, there is $\{\mathbf{u}_k\}_{k=1}^{p} \subset L^2(\Omega; \mathbb R^m)$ satisfying that
\begin{equation}\label{Def-3-1}
  \|\mathbf{y}(T; \mathbf{y}_0, \{\tau_k\}_{k=1}^{p}, \{\mathbf{u}_k\}_{k=1}^{p}) - \mathbf{y}_1\| \leq \varepsilon.
\end{equation}
(ii) The  system (\ref{heat-eq-ip}) is said to be approximately controllable, if for each $T>0$, it is
approximately controllable over $[0,T]$.

\noindent
(iii) We say that the approximate controllability of (\ref{heat-eq-ip}) over $[0,T]$ (with
 $T>0$) can be  realized at  $\{\tau_k\}_{k=1}^p$ (with $p\in\mathbb N^+$ and $0<\tau_1<\cdots<\tau_p<T$), if for any $\varepsilon>0$ and any
   $\mathbf{y}_0, \mathbf{y}_1 \in L^2(\Omega; \mathbb R^n)$, there is $\{\mathbf{u}_k\}_{k=1}^{p} \subset L^2(\Omega; \mathbb R^m)$ satisfying that
\begin{equation}\label{Def-3-2}
  \|\mathbf{y}(T; \mathbf{y}_0, \{\tau_k\}_{k=1}^{p}, \{\mathbf{u}_k\}_{k=1}^{p}) - \mathbf{y}_1\| \leq \varepsilon.
\end{equation}
\end{definition}
Recall that Kalman's controllability rank condition for a  pair $(A,B)$ (in $\mathbb{R}^{n\times n}\times \mathbb{R}^{n\times m}$) is as follows:
\begin{equation}\label{th-1-1}
  \mbox {rank}\;(B, AB, A^2B, \ldots, A^{n-1}B) = n.
\end{equation}

The main results of this paper are presented by the following two  theorems.
 The first one concerns with the null controllability for the system (\ref{heat-eq-ip}),
 while the second one is about the approximate controllability for the system (\ref{heat-eq-ip}).

\begin{theorem}\label{th-1}
 The following  conclusions are true:

(i) When $\Omega \setminus \overline{\omega} \neq \emptyset$ (where $\overline{\omega}$
is the closure of $\omega$ in $\mathbb{R}^N$),   the system (\ref{heat-eq-ip}) is not null controllable over $[0,T]$ for any $T>0$.

(ii) When $\omega = \Omega$,  the system (\ref{heat-eq-ip}) is null controllable if and only if $(A, B)$ satisfies Kalman's controllability rank condition (\ref{th-1-1}).
\end{theorem}

\begin{theorem}\label{th-2}
The following  conclusions are true:

(i)  The system (\ref{heat-eq-ip}) is approximately controllable  if and only if the pair $(A,B)$ satisfies Kalman's controllability rank condition (\ref{th-1-1}).

(ii) Suppose that the pair $(A,B)$ satisfies Kalman's controllability rank condition (\ref{th-1-1}). Then for each $T>0$, the approximate controllability of the system (\ref{heat-eq-ip}) over $[0,T]$ can be realized at any sequence $\{\tau_k\}_{k=1}^n$ with $0 < \tau_1<\cdots<\tau_n < T$ and with $\tau_n-\tau_1<\mbox{d}_A$ (given by (\ref{th-o-1-1}), where $C=A$).
\end{theorem}

\bigskip
Several notes  are given in order.
\begin{itemize}
     \item[(a)] From Theorem \ref{th-1}, we see that the system (\ref{heat-eq-ip})
     does not hold  the null controllability
     except for the  special case when the control region $\omega$ is the whole physical domain $\Omega$. Thus, for the system (\ref{heat-eq-ip}), the approximate controllability is the most likely outcome for us. Fortunately, Theorem \ref{th-2}
     provides a criterion on the approximate controllability for (\ref{heat-eq-ip}). It is exactly  Kalman's controllability rank condition (\ref{th-1-1}).

     For single impulse controlled heat equation, i.e., $n=1$, the approximate controllability can be easily obtained by the qualitative unique continuation at one time point for heat equations (see, for instance, \cite{Lin} for such unique continuation). Moreover, in this case, the approximate controllability can be realized at only one control instant. In \cite{PW1} and \cite{PWZ} (see also
         \cite{AWZ}, \cite{ENZhang} and  \cite{PW}), a quantitative version for such unique continuation was built up. Such  a quantitative version leads to not only the approximate
         controllability but also the approximate null controllability with a cost (see \cite{PW2}).

         For the impulse controlled system (\ref{heat-eq-ip}), we have not found any result on the controllability in past publications.

     \item[(b)]
        The exact controllability was studied in \cite{LXZ} (see also \cite{Yang})
        for  the following impulse controlled linear time-invariant ODE:
         \begin{eqnarray}\label{od-1}
\left\{\begin{array}{llll}
        \frac{d}{dt}z = Az,  &  & t\in \mathbb R^+\setminus\{\tau_k\}_{k=1}^{p}, \\
        z(\tau_k) =z(\tau_k-) + Bu_k, &\;& k = 1,2,\ldots,p,\\
       \end{array}
\right.
\end{eqnarray}
where  $A\in \mathbb{R}^{n\times n}$, $B\in \mathbb{R}^{n\times m}$, $p \in \mathbb N^+$, $\{u_k\}_{k=1}^p\subset \mathbb R^m$ and $\{\tau_k\}_{k=1}^p \subset \mathbb R^+$  is an increasing sequence.
  Let us first recall the following definition of the exact controllability for this system (see   \cite[Definition 2.3.1]{Yang}):  {\it  For each $T>0$ and each $z_0, z_1 \in \mathbb R^n$, there exists $p\in\mathbb N^+$, $\{\tau_k\}_{k=1}^p$ and $\{u_k\}_{k=1}^p$ so that the  corresponding  solution of (\ref{od-1}) drives  $z_0$ at
         $t=0$ to $z_1$ at $t=T$. We say that the exact controllability for (\ref{od-1})
         over $[0,T]$ can be realized at $\{\tau_k\}_{k=1}^p\subset (0,T)$, if for any $z_0, z_1 \in \mathbb R^n$, there exists  $\{u_k\}_{k=1}^p$ so that the
         corresponding  solution of (\ref{od-1}) drives  $z_0$ at
         $t=0$ to $z_1$ at $t=T$. } It was obtained in  \cite[Theorem 1]{LXZ}
          (see also \cite[Theorem 2.3.2]{Yang})
         that the pair $(A,B)$ satisfies Kalman's controllability rank condition (\ref{th-1-1}) if and only if the  exact controllability
          holds. Furthermore, the author in \cite{LXZ} claimed that when $(A,B)$ satisfies (\ref{th-1-1}), the number of control instants  can be taken as the smallest  integer which is bigger than or equals to $n/m$. Unfortunately, we do not understand the proof of this part. (More precisely, we do not understand the argument from Lines 8-9 on Page 83 in \cite{LXZ}.)

         From perspective of control instants, the main difference of
         \cite[Theorem 1]{LXZ} from  our result in (ii) of  Theorem \ref{th-2} is as follows: The author in \cite[Theorem 1]{LXZ} only got the existence  of control instants $\{\tau_k\}_{k=1}^p\subset (0,T)$ at which  the exact controllability
         of (\ref{od-1}) over $[0,T]$ can be realized, but did not know positions of  these control instants. In our Theorem \ref{th-2}, the approximate controllability of the system (\ref{heat-eq-ip})  over $[0,T]$ can be realized at any increasing control instants $\{\tau_k\}_{k=1}^n\subset (0,T)$, provided that $\tau_n-\tau_1<d_A$ (which is given by (\ref{th-o-1-1}) with  $C=A$). Moreover,
         we showed, via  Example \ref{e-2}, that for some $(A,B)\in\mathbb R^{2\times 2}\times \mathbb R^{2\times 1}$ with (\ref{th-1-1}), if an increasing sequence        $\{\tau_k\}_{k=1}^2\subset (0,T)$ satisfies that  $\tau_2-\tau_1=d_A$, then  the approximate controllability
         of  the system  (\ref{heat-eq-ip}) over $[0,T]$ cannot be realized at $\{\tau_k\}_{k=1}^2$. Hence, this condition has certain rationality. Unfortunately, we don't know what happen if $\tau_n-\tau_1>d_A$. Though our result
         on control instants improves greatly that in \cite[Theorem 1]{LXZ}, some  idea in our proof of this result is borrowed from the proof of \cite[Theorem 1]{LXZ}.

      \item[(c)] Two interesting questions are as follows: First, given $(A,B)$ with  (\ref{th-1-1}), what is the least number $p$ of control instants in $(0,T)$ so that the approximate controllability for the system (\ref{heat-eq-ip}) over $[0,T]$
        can be realized at $p$ control instants? Second, what can we say about the approximate controllability for the system (\ref{heat-eq-ip}) where elements of $A$ and $B$ are functions of space variable $x\in\Omega$? So far, we are not able to answer these questions.

        For the first question, Example \ref{e-2} shows that for some $(A,B)\in \mathbb{R}^{2\times 2}\times \mathbb{R}^{2\times 1}$, with (\ref{th-1-1}), the approximate controllability for (\ref{heat-eq-ip}) over $[0,T]$, with $T>0$, cannot be realized at a single control instant $\tau\in (0,T)$.

     \item[(d)] There have been many studies on the the approximate controllability, null controllability and the unique continuation for parabolic equations with controls (or observations) on intervals. Here we would like to mention the following publications and the references therein:
\cite{AK1,AK2,BAR,AWZ,Coron, Xuzhang,
FPZ,FC,FZ,FI,BGM,GH,Lin,LX14,ZX16,lvqi,PW,Wang,CZ-1,LP,ZUA}.
 About works on impulsive controlled systems, we would like to mention the references
\cite{Bensoussan,GZ1,GZ2,LBS,LLZ,LXZ,PYF,XG,Yang,Yong,ZS,ZS2,LXZ2} and the references therein.
\end{itemize}

The rest of the paper is organized as follows: Section 2 proves an important property. Section 3 provides some preliminaries. Section 4 proves Theorem \ref{th-1}. Section 5 shows Theorem \ref{th-2}.

\section{Controllability for impulse controlled  ODEs}

In this section, we will study some properties on the exact  controllability of the system (\ref{od-1}).
 Recall the note $(b)$ in Section 1 for the definition of the  exact controllability of (\ref{od-1}) given in \cite[Definition 2.3.1]{Yang}.
Two main theorems will be introduced.

The first main result of this section is the next Theorem~\ref{th-o-1}, which is one of the bases to prove Theorem \ref{th-2}.
\begin{theorem}\label{th-o-1}
 Let  $(A, B)\in \mathbb{R}^{n\times n}\times \mathbb{R}^{n\times m}$ satisfy  (\ref{th-1-1}). Let $\mbox{d}_A$ be given by (\ref{th-o-1-1}) with $C=A$. Then for each increasing sequence $\{\tau_k\}_{k=1}^n \subset \mathbb R$ with $\tau_n-\tau_1<\mbox{d}_A$, it stands that
\begin{equation}\label{th-o-1-2}
  \mbox{rank}\; (e^{A\tau_1}B, \dots, e^{A\tau_n}B) = n.
\end{equation}
\end{theorem}

The second main result in this section is the following Theorem~\ref{theoremwang2.2}, which will not be used in the proofs of  our main results of the current paper. However, it is independently interesting.
%Before stating Theorem~\ref{theoremwang2.2}, we give the following definition of the above mentioned exact controllability.
%We mention that there is a difference between our definition and \cite[Definition 2.3.1]{Yang}). That is, the control instants in our definition are independent of the initial data and the target, while those in \cite[Definition 2.3.1]{Yang}) are not.%
%We give the following definition, which is a strong version of the definition given by the note $(b)$ in Section 1 for the definition of the  exact controllability of (\ref{od-1}) (i.e. \cite[Definition 2.3.1]{Yang}).
%
%\begin{definition}\label{Def-2.2}
%Let  $(A, B)\in \mathbb{R}^{n\times n}\times \mathbb{R}^{n\times m}$, $T>0$ and $\{\tau_k\}_{k=1}^{p} \subset (0, T)$, with $p \in \mathbb N^+$, be an increasing sequence. We say that the exact controllability for (\ref{od-1})  can be realized at $\{\tau_k\}_{k=1}^p$, if for each  $z_0, z_1 \in \mathbb R^n$, there is $\{u_k\}_{k=1}^{p} \subset \mathbb R^m$   so that corresponding  solution of (\ref{od-1}) drives  $z_0$ at $t=0$ to $z_1$ at $t=T$.
%\end{definition}
We would like to mention that the result in Theorem~\ref{theoremwang2.2} was claimed, without proof, in
  \cite[Theorem 1]{LXZ}  (also in  \cite[Theorem 2.3.2]{Yang}). (See the proof of   \cite[Theorem 1]{LXZ} or \cite[Theorem 2.3.2]{Yang}.)
\begin{theorem}\label{theoremwang2.2}
Let  $(A, B)\in \mathbb{R}^{n\times n}\times \mathbb{R}^{n\times m}$ and  $T>0$. Let $\{\tau_k\}_{k=1}^p\subset (0,T)$, with $p\in \mathbb{N}^+$, be an increasing sequence. Then the exact controllability for (\ref{od-1})
         over $[0,T]$ can be realized at $\{\tau_k\}_{k=1}^p$ if and only if
         \begin{equation}\label{wanggsgs2.2}
  \mbox{rank}\; (e^{A(T-\tau_1)}B, \dots, e^{A(T-\tau_p)}B)=n.
\end{equation}
\end{theorem}
The proofs of the above two theorems will be given later. The following result is a direct consequence of the above two theorems:
\begin{corollary}\label{corollarywang2.3}
Let  $(A, B)\in \mathbb{R}^{n\times n}\times \mathbb{R}^{n\times m}$ satisfy  (\ref{th-1-1}). Let $T>0$.  Then for each increasing sequence $\{\tau_k\}_{k=1}^n \subset (0,T)$, with $\tau_n-\tau_1<\mbox{d}_A$, the exact controllability for (\ref{od-1})
         over $[0,T]$ can be realized at $\{\tau_k\}_{k=1}^n$.
\end{corollary}

The next Example~\ref{e-1} explains the rationality of the condition that $\tau_n-\tau_1<d_A$ in Theorem~\ref{th-o-1}.
\begin{example}\label{e-1}
We will present a pair $(A,B)\in \mathbb{R}^{2\times 2}\times \mathbb{R}^{2\times 1}$ with (\ref{th-1-1})  so that (\ref{th-o-1-2}) is not true for any $\tau_1$ and $\tau_2$, with  $\tau_2-\tau_1=\mbox{d}_A$. For this purpose, we let
 \begin{equation}\label{wanggssg2.3}
 A=\left(\begin{array}{cc}
                          a & -b \\
                          b & a\\
                        \end{array}
                      \right)
\;\;\mbox{and}\;\;
B=\left(
                        \begin{array}{c}
                          c \\
                          d \\
                        \end{array}
                      \right)
                      \end{equation}
with
\begin{equation}\label{e1-24}
  a,b,c,d \in \mathbb R \;\;\mbox{and}\;\; b (c^2 + d^2) \neq 0.
\end{equation}
One can directly check that
\begin{equation*}
  d_A = \frac{\pi}{|b|} \;\;\mbox{and}\;\;
  \mbox{rank}\; (B, AB)=2\;\; (\mbox{since}\; |(B, AB)| = b(c^2 + d^2) \neq 0).
\end{equation*}
Hence, $(A,B)$ satisfies (\ref{th-1-1}).
Meanwhile, we can easily verify  that
\begin{equation*}
  e^{At}=\left(
                        \begin{array}{cc}
                          e^{at}\cos\;bt & -e^{at}\sin\;bt \\
                          e^{at}\sin\;bt & e^{at}\cos\;bt\\
                        \end{array}
                      \right),
  \;\; t\in \mathbb R.
\end{equation*}
Thus, we have  that for any $\tau_1, \tau_2 \in \mathbb R$,
\begin{equation}\label{e1-1}
  |(e^{A\tau_1}B, e^{A\tau_2}B)| = e^{a(\tau_1 + \tau_2)}(c^2 + d^2)\sin b(\tau_2-\tau_1).
\end{equation}
By (\ref{e1-24}) and (\ref{e1-1}), we see that
\begin{equation*}
 \mbox{rank}\;(e^{A\tau_1}B, e^{A\tau_2}B) =2 \Longleftrightarrow b(\tau_2-\tau_1) \notin \{k\pi\; :\;\; k\in\mathbb N \}.
\end{equation*}
Therefore, when $\tau_1, \tau_2 \in \mathbb R$ with $\tau_2-\tau_1=\frac{\pi}{|b|}$, we have that
\begin{equation*}
  \mbox{rank}\;(e^{A\tau_1}B, e^{A\tau_2}B) <2.
\end{equation*}

Besides, this example shows that for each $T>0$, the exact controllability for the system (\ref{od-1}) (with $(A,B)$
 given by (\ref{wanggssg2.3})) over $[0,T]$ cannot be realized at any single control instant. Indeed, this follows at once from Theorem~\ref{theoremwang2.2}
 and the fact that  $\mbox{rank}\;(e^{A\tau}B) <2$ for all $\tau\in \mathbb{R}$.
\end{example}

\begin{remark}
 (i) It was obtained in  \cite[Theorem 1]{LXZ}
          (see also \cite[Theorem 2.3.2]{Yang})
         that if $(A,B)$ satisfies  (\ref{th-1-1}), then for each $T>0$, there is
         $p\in \mathbb{N}^+$ and an increasing sequence $\{\tau_k\}_{k=1}^p\subset (0,T)$ so that
          the exact controllability for (\ref{od-1})
         over $[0,T]$ can be realized at $\{\tau_k\}_{k=1}^p$. However,
         \cite[Theorem 1]{LXZ} does not provide the positions of  $\{\tau_k\}_{k=1}^p$.
         Our Corollary~\ref{corollarywang2.3} shows that  if $(A,B)$ satisfies  (\ref{th-1-1}), then for each $T>0$, the exact controllability for (\ref{od-1})
         over $[0,T]$ can be realized at any increasing sequence $\{\tau_k\}_{k=1}^n\subset (0,T)$, with $\tau_n-\tau_1<d_A$. And our  Example~\ref{e-1} shows  the rationality of the condition that $\tau_n-\tau_1<d_A$.
         Though our Corollary~\ref{corollarywang2.3}
          improves a lot  \cite[Theorem 1]{LXZ} from perspective of control instants,
          some important idea in \cite[Theorem 1]{LXZ} helps us greatly.

          The author in \cite{LXZ} further  claimed that (see \cite[Theorem 1]{LXZ}, or    \cite[Theorem 2.3.2]{Yang})
           when $(A,B)$ satisfies (\ref{th-1-1}), the number of control instants  can be taken as the smallest  integer which is bigger than or equals to $n/m$. Unfortunately, we do not understand the proof of this part. More precisely, we do not understand the argument from Lines 8-9 on Page 83 in \cite{LXZ}.

         (ii) It deserves to mention several facts on  the exact controllability for (\ref{od-1}) where $(A,B)$ satisfies  (\ref{th-1-1}). Fact one: In general,
          for an arbitrarily fixed $T>0$,  we cannot arbitrarily take an increasing sequence $\{\tau_k\}_{k=1}^n\subset (0,T)$  so that
         the exact controllability for (\ref{od-1})
         over $[0,T]$  can  be realized at any increasing sequence $\{\tau_k\}_{k=1}^n\subset (0,T)$ (see
         Example~\ref{e-1}). Fact two:  If all eigenvalues of $A$ are real, then for each $T>0$,  the exact controllability for (\ref{od-1})
         over $[0,T]$  can  be realized at  any $\{\tau_k\}_{k=1}^n\subset (0,T)$. Indeed, in this case, we have that $d_A=\infty$ (see (\ref{th-o-1-1})). Fact three: When
         $T\leq d_A$, the exact controllability for (\ref{od-1})
         over $[0,T]$  can  be realized at  any increasing sequence $\{\tau_k\}_{k=1}^n\subset (0,T)$.
         %However,
%         the following statement is not true:
%
%         \begin{itemize}
%  \item
%         Let $0<T_1<T_2$. Suppose that
%         the exact controllability for (\ref{od-1})
%         over $[0,T_1]$  can  be realized at   $\{\tau_k\}_{k=1}^n\subset (0,T_1)$.
%         Then the exact controllability for (\ref{od-1})
%         over $[0,T_2]$  can  be realized at   $\{\tau_k\}_{k=1}^n$.
%         \end{itemize}
\end{remark}

To prove Theorem \ref{th-o-1}, we need two lemmas. The first lemma presents a kind of decomposition for some  high order ordinary differential operators.

\begin{lemma}\label{lo-2}
Let $C \in \mathbb R^{n \times n}$. Let
   $g(\lambda) = \lambda^n + \sum_{i=0}^{n-1}a_i\lambda^i,\; \lambda \in \mathbb C$ be the characteristic polynomial of the matrix $C$. Let $d_C$ be given by (\ref{th-o-1-1}). Then given $t_0 \in \mathbb R$, there is $\{\phi_i\}_{i=1}^{n}
\subset C^{\infty} \big((t_0 - \frac{d_C}{2}, t_0 + \frac{d_C}{2}); \mathbb R \big)$
so that for each $h \in C^\infty\big((t_0 - \frac{d_C}{2}, t_0 + \frac{d_C}{2}); \mathbb R\big)$,
\begin{eqnarray}\label{lo-2-1}
  g(\frac{d}{dt})h
  = \big(e^{-\phi_1}\circ\frac{d}{dt}\circ e^{\phi_1}\big)\circ\cdots\circ\big(e^{-\phi_n}\circ\frac{d}{dt}\circ e^{\phi_n}\big)h.
\end{eqnarray}
Here, each function $e^{\phi_i}$ (with $i=1,\dots,n$) is regarded as the operator $h \mapsto e^{\phi_i}h$
and the notation $``\circ"$ denotes the composition of operators.
\end{lemma}

\begin{proof}
Arbitrarily fix $t_0 \in \mathbb R$. We first claim the following two facts:
\begin{itemize}
  \item (O1) For each $a \in \mathbb R$,
\begin{equation}\label{lo-2-3}
  \frac{d}{dt} - a = e^{at}\circ\frac{d}{dt}\circ e^{-at}.
\end{equation}
  \item (O2) For any $b, c \in \mathbb R$ with $c \neq 0$, there are
 $\varphi_1$ and $\varphi_2$ in $C^\infty \big( (t_0 - \frac{\pi}{2|c|}, t_0 + \frac{\pi}{2|c|} ); \mathbb R \big)$ so that for each $h \in C^\infty\big( (t_0 - \frac{\pi}{2|c|}, t_0 + \frac{\pi}{2|c|} ); \mathbb R \big)$,
\begin{equation}\label{lo-2-4}
  \big[\frac{d^2}{dt^2} - 2b\frac{d}{dt} + (b^2+c^2)\big]h
  = \big(e^{-\varphi_1}\circ\frac{d}{dt}\circ e^{\varphi_1}\big)\circ\big(e^{-\varphi_2}\circ\frac{d}{dt}\circ e^{\varphi_2}\big)h.
\end{equation}
\end{itemize}
The fact  (O1) can be directly checked. To prove  (O2), we define two functions by
\begin{eqnarray*}
\left.
\begin{array}{ll}
  \varphi_1(t) &\triangleq \displaystyle\int_{t_0}^t \big[-b-c\tan c(\tau-t_0)\big]dt,
  \;\; t \in \Big(t_0-\frac{\pi}{2|c|}, t_0+\frac{\pi}{2|c|} \Big),\\
  \varphi_2(t) &\triangleq \displaystyle\int_{t_0}^t \big[-b+c\tan c(\tau-t_0)\big]dt,
  \;\; t \in \Big(t_0-\frac{\pi}{2|c|}, t_0+\frac{\pi}{2|c|} \Big).
\end{array}
\right.
\end{eqnarray*}
It is clear that $\varphi_1$ and $\varphi_2$ are in $C^\infty \big( (t_0-\frac{\pi}{2|c|}, t_0+\frac{\pi}{2|c|} ); \mathbb R \big)$. By direct computations, we see that for each $t \in \big(t_0-\frac{\pi}{2|c|}, t_0 + \frac{\pi}{2|c|} \big)$,
\begin{eqnarray}\label{wanggsgs2.9}
\varphi'_1(t)+\varphi'_2(t)=-2b \;\;\mbox{and}\;\;
\varphi'_1(t)\varphi'_2(t)+\varphi''_2(t)=b^2+c^2.
\end{eqnarray}
Meanwhile, one can easily check  that for each $h \in C^{\infty}\big( (t_0 - \frac{\pi}{2|c|}, t_0+\frac{\pi}{2|c|} ); \mathbb R \big)$ and each $t \in \big(t_0-\frac{\pi}{2|c|}, t_0+\frac{\pi}{2|c|} \big)$,
\begin{eqnarray}\label{lo-2-16}
&&\Big(e^{-\varphi_1}\circ\frac{d}{dt}\circ e^{\varphi_1}\Big)\circ\Big(e^{-\varphi_2}\circ\frac{d}{dt}\circ e^{\varphi_2}\Big)h(t)
\nonumber\\
&=& \frac{d^2}{dt^2}h(t)+\big(\varphi'_1(t)+\varphi'_2(t)\big)\frac{d}{dt}h(t)
  + \big(\varphi'_1(t)\varphi'_2(t)+\varphi''_2(t)\big)h(t).
\end{eqnarray}
Now, (\ref{lo-2-4}) follows from (\ref{wanggsgs2.9}) and  (\ref{lo-2-16}), i.e.,  the
fact (O2) is true.

Next, since $g(\cdot)$ is a  polynomial with real coefficients (which implies that $g(\bar{\lambda}_0)=0$ if and only if $g(\lambda_0)=0$),  we can write all of its roots (i.e., the solutions of $g(\lambda)=0$) in the following manner:
\begin{equation*}
\alpha_1, \dots, \alpha_{n_1}\;(\mbox{real}),\; \; \beta_1, \bar{\beta}_1, \dots, \beta_{n_2},  \bar{\beta}_{n_2}\; (\mbox{non-real}), \;\;\mbox{with}\;\; n_1+2n_2=n.
\end{equation*}
(Here, it is allowed that the multiplicity of $\alpha_i$ (or $\beta_j$) is bigger than $1$. i.e., it may happen that $\alpha_1=\alpha_2$ (or $\beta_1=\beta_2$).)
 There are only three possibilities on $n_1$: (i) $n_1=0$; (ii) $n_1=n$; (iii) $1\leq n_1<n$.

Case\;(i): We have that $n_1=0$ and $n=2n_2$. Then
\begin{eqnarray}\label{lo-2-15-1}
  g(\frac{d}{dt})
  =\prod_{i=1}^{n_2}(\frac{d}{dt}-\beta_i\big)\big(\frac{d}{dt}-\overline{\beta}_i)
  =\prod_{i=1}^{n_2}\big(\frac{d^2}{dt^2}-2\mbox{Re}\; \beta_i\frac{d}{dt}
   + |\beta_i|^2\big).
\end{eqnarray}
For each $i \in \{1,\ldots, n_2\}$, we can apply the fact (O2), where
$b=\mbox{Re}\; \beta_i$ and $c=\mbox{Im}\; \beta_i$,
to find  $\phi_{i_1}$ and $\phi_{i_2}$ in $C^\infty \big( (t_0-\frac{\pi}{2|\mbox{Im}\;\beta_i|}, t_0+\frac{\pi}{2|\mbox{Im}\;\beta_i|}); \mathbb R \big)$ so that for each $h \in C^\infty \big( (t_0-\frac{\pi}{2|\mbox{Im}\;\beta_i|}, t_0+\frac{\pi}{2|\mbox{Im}\;\beta_i|}); \mathbb R \big)$,
\begin{equation}\label{lo-2-17}
  \big[\frac{d^2}{dt^2} - 2\mbox{Re}\; \beta_i\frac{d}{dt}
   + |\beta_i|^2\big]h
  =\big[(e^{-\phi_{i_1}}\circ\frac{d}{dt}\circ e^{\phi_{i_1}})\circ
  (e^{-\phi_{i_2}}\circ\frac{d}{dt}\circ e^{\phi_{i_2}})\big]h.
\end{equation}
Meanwhile, from the definition of $d_C$ (see (\ref{th-o-1-1})), it follows that $d_C \leq \frac{\pi}{|\mbox{Im}\;\beta_i|}$ for each $i \in \{1, \cdots, n_2\}$. This, along with (\ref{lo-2-15-1}) and (\ref{lo-2-17}), leads to (\ref{lo-2-1}) in Case (i).

 Case\;(ii): We have that $n_1=n$ and $n_2=0$. Then
\begin{eqnarray}\label{lo-2-15-2}
  g(\frac{d}{dt})
  = \prod_{i=1}^n\big(\frac{d}{dt} - \alpha_i\big).
\end{eqnarray}
For each $i \in \{1,\ldots, n\}$, we can apply the fact (O1), where $a = \alpha_i,$ to find $\phi_i \in C^\infty(\mathbb R; \mathbb R)$ so that
\begin{equation*}
  \frac{d}{dt} - \alpha_i
  = e^{-\phi_i}\circ\frac{d}{dt}\circ e^{\phi_i}.
\end{equation*}
This, along with (\ref{lo-2-15-2}), leads to (\ref{lo-2-1}) in Case (ii).

 Case\;(iii): We have that $n_1 \geq 1$ and $n_2 \geq 1$. Then
\begin{eqnarray}\label{lo-2-15-3}
  g(\frac{d}{dt})
  &=& \Big[\prod_{i=1}^{n_1}\big(\frac{d}{dt} - \alpha_i\big)\Big]
  \Big[\prod_{j=1}^{n_2}\big(\frac{d}{dt}-\beta_j\big)\big(\frac{d}{dt}-\overline{\beta}_j\big)\Big]
  \nonumber\\
  &=&\Big[\prod_{i=1}^{n_1}\big(\frac{d}{dt} - \alpha_i\big)\Big]
   \Big[\prod_{j=1}^{n_2}\big(\frac{d^2}{dt^2} - 2\mbox{Re}\; \beta_j\frac{d}{dt}
   + |\beta_j|^2\Big].
\end{eqnarray}
For each $i \in \{1,\ldots, n_1\}$ and each $j \in \{1,\ldots, n_2\}$, we can  apply respectively
 the facts  (O1) where $a=\alpha_i$ and (O2) where $b=\mbox{Re}\; \beta_j$ and
 $c=\mbox{Im}\; \beta_j$, to
 find that  $\phi_i \in C^\infty(\mathbb R; \mathbb R)$ and $\phi_{j_1}, \phi_{j_2} \in C^\infty \Big( \big(t_0-\frac{\pi}{2|\mbox{Im}\; \beta_j|}, t_0+\frac{\pi}{2|\mbox{Im}\; \beta_j|} \big); \mathbb R \Big)$,
so that for each $h \in C^\infty \big( (t_0-\frac{\pi}{2|\mbox{Im}\;\beta_j|}, t_0+\frac{\pi}{2|
\mbox{Im}\;\beta_j|}); \mathbb R \big)$,
\begin{eqnarray}\label{lo-2-15-4}
\left.
\begin{array}{ll}
  ~~~~~~~~~~~~~~~~\big(\frac{d}{dt} - \alpha_i\big)h
  &= \big(e^{-\phi_i}\circ\frac{d}{dt}\circ e^{\phi_i}\big)h,
  \\
  \big[\frac{d^2}{dt^2} - 2\mbox{Re}\; \beta_j\frac{d}{dt}
   + |\beta_j|^2\big]h
  &=\big[(e^{-\phi_{j_1}}\circ\frac{d}{dt}\circ e^{\phi_{j_1}})\circ
  (e^{-\phi_{j_2}}\circ\frac{d}{dt}\circ e^{\phi_{j_2}})\big]h.
  \end{array}
\right.
\end{eqnarray}
Meanwhile, from the definition of $d_C$ (see (\ref{th-o-1-1})), it follows that $d_C \leq \frac{\pi}{|\mbox{Im}\;\beta_j|}$ for each $j \in \{1, \cdots, n_2\}$. This, along with (\ref{lo-2-15-3}) and (\ref{lo-2-15-4}), leads to (\ref{lo-2-1}) in Case (iii).

\vskip 5pt
In summary, we end the proof of Lemma~\ref{lo-2}.
\end{proof}

The following lemma presents a kind of unique continuation for some high order ordinary differential equations.

\begin{lemma}\label{lo-1}
Let $C \in \mathbb R^{n \times n}$. Let $g(\lambda) = \lambda^n + \sum_{i=0}^{n-1}a_i\lambda^i,\; \lambda \in \mathbb C$ be the characteristic polynomial of the matrix $C$. Let $\mbox{d}_C$ be given by (\ref{th-o-1-1}).  Suppose that  $f \in C^{\infty}(\mathbb R; \mathbb R)$ satisfies   that
\begin{equation}\label{lo-1-2}
  g(\frac{d}{dt})f = 0 \;\; \mbox{over}\; \mathbb R
  \;\;\mbox{and}\;\;
  f(\tau_k) = 0, \;\; \forall\; k \in \{1, \ldots, n\},
\end{equation}
for some increasing sequence $\{\tau_k\}_{k=1}^n \subset \mathbb R$, with $\tau_n - \tau_1 < \mbox{d}_C$.
Then
\begin{equation}\label{lo-1-3}
  f \equiv 0 \;\; \mbox{over}\; \mathbb R.
\end{equation}
\end{lemma}

\begin{proof}
Let $f$ (in $C^{\infty}(\mathbb R; \mathbb R)$) and an increasing sequence $\{\tau_i\}_{i=1}^n$ (in $\mathbb R$), with $\tau_n-\tau_1<d_C$, satisfy (\ref{lo-1-2}). We aim to show that $f$ satisfy (\ref{lo-1-3}).
For this purpose, we arbitrarily fix  $\hat t_0\in\mathbb R$ so that
\begin{equation}\label{lo-16-1}
  \{\tau_i\}_{i=1}^n \subset \Big(\hat t_0-\frac{d_C}{2}, \hat t_0+\frac{d_C}{2}\Big).
\end{equation}
(Such $\hat t_0$ exists because $\tau_n-\tau_1<d_C$.)
According to Lemma \ref{lo-2},  there is a sequence $\{\phi_i\}_{i=1}^n$ in $ C^\infty\big((\hat t_0-\frac{d_C}{2}, \hat t_0+\frac{d_C}{2}); \mathbb R \big)$ so that for each $h \in C^\infty\big((\hat t_0-\frac{d_C}{2}, \hat t_0+\frac{d_C}{2}); \mathbb R \big)$,
\begin{equation}\label{lo-1-4}
   g(\frac{d}{dt})h
  = \big(e^{-\phi_n}\circ\frac{d}{dt}\circ e^{\phi_n}\big)\circ\cdots\circ\big(e^{-\phi_1}\circ\frac{d}{dt}\circ e^{\phi_1}\big)h.
\end{equation}

                When $n=1$, we have that $g(\frac{d}{dt}) = \frac{d}{dt} + a_0$,
                which, along with  (\ref{lo-1-2}), yields that $f'(t)+a_0f(t)=0$ for all $t\in \mathbb{R}$; and that $f(\tau_1)=0$. Hence, $f\equiv 0$, i.e.,
                (\ref{lo-1-3}) is true in the case that $n=1$.

                We now show   (\ref{lo-1-3}) for the case when $n\geq 2$.
The  proof  is organized as two steps.

\vskip5pt
\textit{Step 1. We show that for each $k \in \{1, \ldots, n-1\},$ there is an increasing sequence $\{\xi_{k,j}\}_{j=1}^{n-k} \subset (\hat t_0 - \frac{d_C}{2}, \hat t_0 + \frac{d_C}{2})$ so that
\begin{equation}\label{s1-1}
   \Big[\big(e^{-\phi_k}\circ\frac{d}{dt}\circ e^{\phi_k}\big)\circ\cdots\circ\big(e^{-\phi_1}\circ\frac{d}{dt}\circ e^{\phi_1}\big)f\Big](\xi_{k,j}) = 0,
   \;\;\forall\; j\in \{1,\ldots, n-k\}.
\end{equation}
}

First, we show (\ref{s1-1}) with $k=1$. By the second equality in (\ref{lo-1-2}), we find that
\begin{equation*}
   (e^{\phi_1}f)(\tau_i) = 0,
   \;\;\forall\; i\in \{1,\ldots, n\}.
\end{equation*}
From this and the mean value theorem, we deduce that there exists an increasing sequence $\{\xi_{1,j}\}_{j=1}^{n-1} \subset (\tau_1, \tau_n)$ so that
\begin{equation*}
   \big[\frac{d}{dt}(e^{\phi_1}f)\big](\xi_{1,j}) = 0, \;\;\forall\; j\in \{1,\ldots, n-1\},
\end{equation*}
from which, it follows  that
\begin{equation*}
   \big[e^{-\phi_1}\frac{d}{dt}(e^{\phi_1}f)\big](\xi_{1,j}) = 0, \;\;\forall\; j\in \{1,\ldots, n-1\}.
\end{equation*}
Since $(\tau_1, \tau_n) \subset (\hat t_0 - \frac{d_C}{2}, \hat t_0 + \frac{d_C}{2})$ (see (\ref{lo-16-1})), the above yields (\ref{s1-1}) for $k=1$.

Next, we will show (\ref{s1-1}) for each $k \in \{1,\ldots, n-1\}$. Since we are in the case that $n \geq 2$, there are only two possibilities on $n$: either $n=2$ or $n \geq 3$. In the first case that $n=2$, we have that $n-1=1$. Then (\ref{s1-1}) has been proved since $k$ can only take $1$ now.

 In the second case $n\geq 3$, we will  show (\ref{s1-1}) by using mathematical induction.  We already have (\ref{s1-1}) with $k=1$.
Suppose that (\ref{s1-1}) holds for $k=m$, with  $m<n-1$. That is,  there exists an increasing sequence $\{\xi_{m,j}\}_{j=1}^{n-m} \subset (\hat t_0-\frac{d_C}{2}, \hat t_0+\frac{d_C}{2})$ so that
\begin{equation}\label{s1-2}
   \Big[\big(e^{-\phi_m}\circ\frac{d}{dt}\circ e^{\phi_m}\big)\circ\cdots\circ\big(e^{-\phi_1}\circ\frac{d}{dt}\circ e^{\phi_1}\big)f\Big](\xi_{m,j}) = 0,
   \;\;\forall\; j\in \{1,\ldots, n-m\}.
\end{equation}
We aim to prove (\ref{s1-1}) with $k= m+1$. For this purpose, we set
\begin{equation}\label{s1-3}
  q_m(t) \triangleq \big(e^{-\phi_m}\circ\frac{d}{dt}\circ e^{\phi_m}\big)\circ\cdots\circ\big(e^{-\phi_1}\circ\frac{d}{dt}\circ e^{\phi_1}\big)f(t),\;\;
  t \in \Big(\hat t_0-\frac{d_C}{2}, \hat t_0+\frac{d_C}{2}\Big).
\end{equation}
From (\ref{s1-3}) and (\ref{s1-2}), it follows that
\begin{equation*}
  e^{\phi_{m+1}}q_m(\xi_{m,j}) = 0, \;\;\forall\; j \in \{1, \ldots, n-m\}.
\end{equation*}
By this and the mean value theorem, we find that there exists an increasing sequence $\{\xi_{m+1,j}\}_{j=1}^{n-m-1} \subset (\hat t_0 - \frac{d_C}{2}, \hat t_0 + \frac{d_C}{2})$ so that
\begin{equation*}
  \frac{d}{dt}\big(e^{\phi_{m+1}}q_m\big)(\xi_{m+1,j}) = 0, \;\;\forall\; j \in \{1, \ldots, n-m-1\},
\end{equation*}
from which, it follows  that
\begin{equation*}
 e^{-\phi_{m+1}} \frac{d}{dt}\big(e^{\phi_{m+1}}q_m\big)(\xi_{m+1,j}) = 0, \;\;\forall\; j \in \{1, \ldots, n-m-1\}.
\end{equation*}
This, along with (\ref{s1-3}), leads to (\ref{s1-1}) with $k = m+1$. In summary, we conclude that  (\ref{s1-1}) is true. This ends the proof of Step 1.

\vskip5pt
\textit{Step 2. We show (\ref{lo-1-3}).}

We first claim that for each $k \in \{1, \ldots, n\},$
\begin{equation}\label{s2-1}
   \big(e^{-\phi_k}\circ\frac{d}{dt}\circ e^{\phi_k}\big)\circ\cdots\circ\big(e^{-\phi_1}\circ\frac{d}{dt}\circ e^{\phi_1}\big)f \equiv 0
   \;\;\mbox{over}\; \Big(\hat t_0-\frac{d_C}{2}, \hat t_0+\frac{d_C}{2}\Big).
\end{equation}
(We will only use (\ref{s2-1}) with $k=1$ later.)
By contradiction, we suppose that (\ref{s2-1}) was not true for some $\hat k \in \{1, \ldots, n\}.$ Then we would have that
\begin{equation}\label{s2-3}
  S \neq \emptyset,\;\;\mbox{where}\;\;S \triangleq \big\{\bar k \in \{1,\ldots,n\} \;:\; (\ref{s2-1}),\;\mbox{with}\;k=\bar k,\;\; \mbox{fails} \big\}.
\end{equation}
By (\ref{lo-1-2}) and (\ref{lo-1-4}), we see that  (\ref{s2-1}), with $k=n$, is true. Hence, $n\notin S$. This, along with (\ref{s2-3}), yields that
\begin{equation}\label{s2-4}
  \hat k_1 \triangleq \max_{\bar k \in S} \bar k <n.
\end{equation}
By (\ref{s2-4})  and the definition of $S$ (see (\ref{s2-3})), we find that (\ref{s2-1}), with $k=\hat k_1+1 (\leq n)$,
is true.  That is,
\begin{equation}\label{s2-5}
   (e^{-\phi_{\hat k_1+1}}\circ\frac{d}{dt}\circ e^{\phi_{\hat k_1+1}}\big)\circ\cdots\circ\big(e^{-\phi_1}\circ\frac{d}{dt}\circ e^{\phi_1}\big)f \equiv 0
   \;\;\mbox{over}\; \Big(\hat t_0 - \frac{d_C}{2}, \hat t_0 + \frac{d_C}{2}\Big).
\end{equation}
Define a function $\hat f$ in the following manner:
\begin{equation}\label{s2-6}
 \hat f(t) \triangleq \big(e^{-\phi_{\hat k_1}}\circ\frac{d}{dt}\circ e^{\phi_{\hat k_1}}\big)\circ\cdots\circ\big(e^{-\phi_1}\circ\frac{d}{dt}\circ e^{\phi_1}\big)f(t),
   \;\; t\in \Big(\hat t_0-\frac{d_C}{2}, \hat t_0+\frac{d_C}{2}\Big).
\end{equation}
By (\ref{s2-6}) and (\ref{s2-5}), we find that
\begin{equation*}
 \big(e^{-\phi_{\hat k_1+1}}\circ\frac{d}{dt}\circ e^{\phi_{\hat k_1+1}}\big)\hat f(t) \equiv 0
   \;\; \mbox{for all}\;\; t\in\Big(\hat t_0-\frac{d_C}{2}, \hat t_0+\frac{d_C}{2}\Big),
\end{equation*}
from which, it follows  that
\begin{equation*}
   \frac{d}{dt}\big(e^{\phi_{\hat k_1+1}(t)}\hat f(t)\big) \equiv 0
   \;\; \mbox{for all}\;\; t\in\Big(\hat t_0-\frac{d_C}{2}, \hat t_0+\frac{d_C}{2}\Big).
\end{equation*}
This implies that
\begin{equation}\label{s2-7}
   e^{\phi_{\hat k_1+1}(t)}\hat f(t) \equiv \;\mbox{const}
   \;\; \mbox{for all}\;\; t\in\Big(\hat t_0-\frac{d_C}{2}, \hat t_0+\frac{d_C}{2}\Big).
\end{equation}
Meanwhile, by (\ref{s2-6}), (\ref{s2-4}) and Step\;1 (where $k = \hat k_1$), we get that
$\hat f(\hat \tau) = 0$ for some $\hat \tau \in (\hat t_0 - \frac{d_C}{2}, \hat t_0 + \frac{d_C}{2}).$ This, along with (\ref{s2-7}), indicates that
\begin{equation*}
  \hat f \equiv 0\;\; \mbox{over}\;\; \Big(\hat t_0-\frac{d_C}{2}, \hat t_0+\frac{d_C}{2}\Big).
\end{equation*}
This, along with  (\ref{s2-6}), leads to (\ref{s2-1}) with $k=\hat k_1$.
 By (\ref{s2-1}), with $k=\hat k_1$,  and by the definition of $S$ (see (\ref{s2-3})), we find that $\hat k_1 \notin S$, which contradicts the definition of $\hat k_1$ (see (\ref{s2-4})). Therefore, (\ref{s2-1}) holds for all $k\in \{1, \ldots, n\}$.

Finally, by (\ref{s2-1}), with $k=1$,  we have that
 \begin{equation*}
   \big(e^{\phi_{1}}\circ\frac{d}{dt}\circ e^{\phi_{1}}\big)f \equiv 0
   \;\; \mbox{over}\; \Big(\hat t_0 - \frac{d_C}{2}, \hat t_0 + \frac{d_C}{2}\Big),
\end{equation*}
 from which, it follows that
 \begin{equation*}
   e^{\phi_{1}} f \equiv \mbox{const}
   \;\; \mbox{over}\; \Big(\hat t_0-\frac{d_C}{2}, \hat t_0+\frac{d_C}{2}\Big).
\end{equation*}
This, along with  the second equality in (\ref{lo-1-2}), indicates that
\begin{equation}\label{lo-1-2-1}
    f \equiv 0
   \;\; \mbox{over}\; \Big(\hat t_0-\frac{d_C}{2}, \hat t_0+\frac{d_C}{2}\Big).
\end{equation}
At same time, since $f$ is a solution to the equation: $g(\frac{d}{dt})f=0$,  we see that $f$ is analytic over $\mathbb R$. This, along with (\ref{lo-1-2-1}), leads to (\ref{lo-1-3}).

In summary, we end the proof of Lemma~\ref{lo-1}.
\end{proof}

Now we are  on the position to prove Theorem \ref{th-o-1}.
\begin{proof}[Proof of Theorem \ref{th-o-1}]
 When $n=1$,  one can easily check, from (\ref{th-1-1}),  that $\mbox{rank}\; B =1$. Then  we find that for each $t>0$,
$\mbox{rank}\; e^{At}B = \mbox{rank}\; B =1,$ which leads to (\ref{th-o-1-2}) for the case that $n=1$.

We now show  (\ref{th-o-1-2}) for the case that $n > 1$. Let $g(\lambda) = \lambda^n + \sum_{i=0}^{n-1}a_i\lambda^i,\; \lambda \in \mathbb C$,  be the characteristic polynomial of the matrix $A$. Write
\begin{eqnarray}\label{to-16-1}
\hat A = \left(
                        \begin{array}{ccccc}
                          0 & 0 & \cdots & 0 & -a_0\\
                          1 & 0 & \cdots & 0 & -a_1\\
                          0 & 1 & \cdots & 0 & -a_2\\
                          \vdots & \vdots & \ddots & \vdots & \vdots\\
                          0 & 0 & \cdots & 1 & -a_{n-1}\\
                        \end{array}
                      \right).
\end{eqnarray}
Consider the equation:
\begin{equation}\label{lm-2-4}
   \frac{d}{dt}  \left(
                        \begin{array}{c}
                          f_0(t)\\
                          \vdots\\
                          f_{n-1}(t)\\
                        \end{array}
                      \right)  = \hat A \left(
                        \begin{array}{c}
                          f_0(t)\\
                          \vdots\\
                          f_{n-1}(t)\\
                        \end{array}
                      \right),
\;\; t \in \mathbb R,
\end{equation}
with the  initial value condition:
\begin{eqnarray}\label{lm-2-5}
  (f_0(0),   f_1(0), \dots,  f_n(0))^\top=(1, 0, \dots, 0)^\top\triangleq e_1.
 \end{eqnarray}
Let $(\hat f_0, \hat f_1, \dots, \hat f_{n-1})^\top \in C^\infty(\mathbb R^1; \mathbb R^n)$ be the solution of (\ref{lm-2-4})-(\ref{lm-2-5}).

The rest proof of this theorem is divided into the following three steps:

\vskip5pt
\textit{Step 1. We  show that
\begin{equation}\label{o-17-1}
  e^{At} = \sum_{i=0}^{n-1}\hat f_i(t)A^i \;\;\mbox{for all}\;\; t \in \mathbb R.
\end{equation}}

Define a matrix-valued function in the following manner:
\begin{equation}\label{lm-2-6}
  \Phi(t) \triangleq \sum_{i=0}^{n-1}\hat f_i(t)A^i \;\;\mbox{for all}\;\; t \in \mathbb R.
\end{equation}
 By (\ref{lm-2-6}), (\ref{lm-2-4}) and (\ref{to-16-1}), we find that for all $t\in\mathbb R$,
\begin{eqnarray}\label{lm-2-7}
  \frac{d}{dt}\Phi(t) = \sum_{i=0}^{n-1}\frac{d}{dt}\hat f_i(t)A^i
 &=& -a_0\hat f_{n-1}(t) + \sum_{i=1}^{n-1}\big(\hat f_{i-1}(t) - a_i\hat f_{n-1}(t)\big)A^i
 \nonumber\\
 &=& \sum_{i=1}^{n-1}\hat f_{i-1}(t)A^i - \Big(\sum_{i=0}^{n-1}a_iA^i\Big)\hat f_{n-1}(t).
\end{eqnarray}
 Since $g(\lambda) = \lambda^n + \sum_{i=0}^{n-1}a_i\lambda^i,\; \lambda \in \mathbb C$, is the characteristic polynomial of the matrix $A$, it follows  by  the Cayley-Hamilton Theorem, (\ref{lm-2-7}) and (\ref{lm-2-6}) that for each $t \in \mathbb R$,
\begin{eqnarray}\label{lm-2-7-1}
  \frac{d}{dt}\Phi(t) &=& \sum_{i=1}^{n-1}\hat f_{i-1}(t)A^i - (g(A) - A^n)\hat f_{n-1}(t)
  = \sum_{i=1}^{n}\hat f_{i-1}(t)A^i
\nonumber\\
  &=& A\sum_{i=1}^{n}\hat f_{i-1}(t)A^{i-1} = A\sum_{i=0}^{n-1}\hat f_i(t)A^i = A\Phi(t).
\end{eqnarray}
Meanwhile, by (\ref{lm-2-6}) and (\ref{lm-2-5}), we find that
\begin{eqnarray}\label{lm-2-8}
  \Phi(0) = I.
\end{eqnarray}
Now, (\ref{o-17-1}) follows from  (\ref{lm-2-7-1}), (\ref{lm-2-8}) and (\ref{lm-2-6}).  This ends the proof of Step\;1.

\textit{Step 2. We prove that for each increasing sequence $\{\tau_k\}_{k=1}^n \subset \mathbb R$, with $\tau_n-\tau_1<d_A,$
\begin{equation}\label{o-1}
  \mbox{rank}\;\left(
                        \begin{array}{cccc}
                          \hat f_0(\tau_1)  & \cdots & \hat f_0(\tau_n)\\
                          \vdots  & \ddots & \vdots\\
                          \hat f_{n-1}(\tau_1) & \cdots & \hat f_{n-1}(\tau_n)\\
                        \end{array}
                      \right)=n.
\end{equation}}

Arbitrarily take an increasing sequence $\{\tau_k\}_{k=1}^n \subset \mathbb R$ so that
\begin{equation}\label{o-17-2}
  \tau_n-\tau_1<d_A.
\end{equation}
To show (\ref{o-1}), it suffices to prove that the following equation
has a unique solution $\mathbf{x}=0$ in $\mathbb R^n$:
\begin{equation}\label{o-2}
  \mathbf{x}^\top \left(
                        \begin{array}{cccc}
                          \hat f_0(\tau_1) & \cdots & \hat f_0(\tau_n)\\
                          \vdots & \ddots & \vdots\\
                          \hat f_{n-1}(\tau_1) & \cdots & \hat f_{n-1}(\tau_n)\\
                        \end{array}
                      \right) = 0.
\end{equation}
 For this purpose, we  let $\mathbf{x} \in \mathbb R^n$ be a solution to  (\ref{o-2}).
Since $(\hat f_0, \hat f_1, \dots, \hat f_{n-1})^\top$ solves (\ref{lm-2-4})-(\ref{lm-2-5}), we find from (\ref{o-2}) that
\begin{equation}\label{o-2-1}
  \mathbf{x}^\top e^{\hat A \tau_k}e_1 = \mathbf{x}^\top\left(
                        \begin{array}{c}
                         \hat f_0(\tau_k)\\
                          \vdots\\
                         \hat f_{n-1}(\tau_k)\\
                        \end{array}
                      \right) = 0,
  \;\; \forall \; k \in \{1,\ldots, n\}.
\end{equation}
Let $\hat{\mathbf{h}}\triangleq (\hat h_0, \cdots, \hat h_{n-1})^\top \in C^\infty (\mathbb R; \mathbb R^n)$ satisfy that
 \begin{equation}\label{o-17-6}
   \left\{\begin{array}{ll}
        &\dfrac{d}{dt} \hat{\mathbf{h}}(t)=\hat A^\top \hat{\mathbf{h}}(t),\;\;  t\in \mathbb R,\\
        &\hat{\mathbf{h}}(0)=\mathbf x.
       \end{array}
\right.
 \end{equation}
 Since $g(\lambda)=\lambda^n+\sum_{i=0}^{n-1}a_i\lambda^i,\; \lambda \in \mathbb C$,
  it follows by (\ref{o-17-6}) and (\ref{to-16-1}) that
\begin{equation}\label{o-17-4}
  \hat{\mathbf h}=(\hat h_0, \frac{d}{dt}\hat h_0, \cdots, \frac{d^{n-1}}{dt^{n-1}}\hat h_0)^\top
\;\;\mbox{and}\;\;
g(\frac{d}{dt}) {\hat h}_0 =0.
\end{equation}
Meanwhile, by (\ref{o-17-6}) and (\ref{o-2-1}), we get that for each $k \in \{1,\ldots, n\}$,
\begin{eqnarray}\label{o-5}
  \hat h_0(\tau_k) = \big<\hat {\mathbf h}(\tau_k), e_1\big>_{\mathbb R^n}
  = \big<e^{\hat A^\top \tau_k}\mathbf{x}, e_1\big>_{\mathbb R^n}
  = {\mathbf x}^\top e^{\hat A \tau_k}e_1 =0.
\end{eqnarray}
Since $g(\cdot)$ is the characteristic polynomial of the matrix $A$, by the second equality in (\ref{o-17-4}),
(\ref{o-5}) and (\ref{o-17-2}), we can apply Lemma~\ref{lo-1}, to obtain that
$\hat h_0 \equiv 0$ over $\mathbb R$. This, together with the first equality in (\ref{o-17-4}), yields that
$\hat{\mathbf h} \equiv 0$ over $\mathbb R$,
from which, as well as (\ref{o-17-6}), it follows that  $\mathbf x=0$, i.e., (\ref{o-2}) has a unique solution $\mathbf x=0$.  Hence, (\ref{o-1}) is true.

\vskip5pt
\textit{Step 3. We prove (\ref{th-o-1-2}).}

Arbitrarily take an increasing sequence $\{\tau_k\}_{k=1}^n \subset \mathbb R$ so that
\begin{equation}\label{th-o-17-1}
  \tau_n-\tau_1<d_A.
\end{equation}
 To prove (\ref{th-o-1-2}), it suffices to show that the following equation
 has a unique solution $\alpha=0$:
\begin{equation}\label{lm-2-10}
  \alpha^\top(e^{A\tau_1}B, e^{A\tau_2}B, \ldots, e^{A\tau_{n}}B) = 0.
\end{equation}
 For this purpose, we let  $\alpha\in \mathbb R^n$ satisfy (\ref{lm-2-10}). Then, by Step\;1, we see that for each $k \in \{1, \ldots, n\}$,
\begin{eqnarray*}
&&\big(B^*\alpha, B^*A^*\alpha, \ldots, B^*(A^*)^{n-1}\alpha\big)
\left(
                        \begin{array}{cccc}
                          \hat f_0(\tau_k) \\
                          \vdots \\
                          \hat f_{n-1}(\tau_k) \\
                        \end{array}
                      \right)
\\
&=& \sum_{i=0}^{n-1}\hat f_i(\tau_k)B^*(A^*)^i\alpha
= B^*\big(\sum_{i=0}^{n-1}\hat f_i(\tau_k)A^i\big)^*\alpha
=B^*e^{A^*\tau_k}\alpha = (\alpha^\top e^{A\tau_k}B)^\top = 0.
\end{eqnarray*}
From this, we get that
\begin{equation*}
 \big(B^*\alpha, B^*A^*\alpha, \ldots, B^*(A^*)^{n-1}\alpha\big)
 \left(
                        \begin{array}{cccc}
                          \hat f_0(\tau_1) & \cdots & \hat f_0(\tau_n)\\
                          \vdots & \ddots & \vdots\\
                          \hat f_{n-1}(\tau_1) & \cdots & \hat f_{n-1}(\tau_n)\\
                        \end{array}
                      \right) = 0.
\end{equation*}
This, along with Step\;2 and (\ref{th-o-17-1}), yields that
\begin{equation*}
   B^*(A^*)^{k-1}\alpha=0\;\;\mbox{for each}\;\;k \in \{1, \ldots, n\},
\end{equation*}
from which, it follows  that
\begin{eqnarray*}
   \alpha^\top(B, AB, A^2B, \ldots, A^{n-1}B) &=& (\alpha^\top B, \alpha^\top AB, \alpha^\top A^2B, \ldots, \alpha^\top A^{n-1}B) \\
   &=&
  \left(
                        \begin{array}{cccc}
                          B^*\alpha \\
                         B^*A^*\alpha \\
                          \vdots \\
                         B^*(A^*)^{n-1}\alpha\\
                        \end{array}
                      \right)^\top = 0.
\end{eqnarray*}
Since $(A, B)$ satisfies (\ref{th-1-1}), the above yields that $\alpha=0$. Thus,  (\ref{lm-2-10}) only has the trivial solution in $\mathbb R^n$. Therefore, (\ref{th-o-1-2}) is true.

\vskip 5pt
In summary, we end the proof of Theorem \ref{th-o-1}.
\end{proof}

We are now on the position to prove Theorem~\ref{theoremwang2.2}.
\begin{proof}[Proof of Theorem \ref{theoremwang2.2}]
Let $T>0$. Let $\{\tau_k\}_{k=1}^p\subset(0,T)$, with $p\in \mathbb{N}^+$, be an increasing sequence. First of all, for each $z_0 \in \mathbb R^n$ and $\{u_k\}_{k=1}^p\subset \mathbb R^m$, the solution of (\ref{od-1}) satisfies that
\begin{eqnarray}\label{od-1-1}
z(T; z_0, \{\tau_k\}_{k=1}^p, \{u_k\}_{k=1}^p) = e^{AT}z_0 + \sum_{1\leq k\leq p} e^{ A(T-\tau_k)}Bu_k,
\end{eqnarray}
where $z(\cdot; z_0,\{\tau_k\}_{k=1}^{p}, \{u_k\}_{k=1}^{p} )$ denotes the solution of (\ref{od-1}).

We next prove the sufficiency. Assume that (\ref{wanggsgs2.2}) is true. Arbitrarily fix $\hat z_0, \hat z_1 \in \mathbb R^n$. By (\ref{wanggsgs2.2}), we find that there exists $\{\hat{ u}_k\}_{k=1}^p \subset \mathbb R^m$ so that
\begin{equation*}
  \sum_{k=1}^p e^{ A(T-\tau_k)}B\hat u_k = \hat z_1 - e^{AT}\hat z_0.
\end{equation*}
From this and (\ref{od-1-1}), we see that
\begin{equation*}
   z(T; \hat z_0, \{\tau_k\}_{k=1}^p, \{\hat u_k\}_{k=1}^p) = \hat z_1.
\end{equation*}
Since $\hat z_0, \hat z_1 $ were arbitrarily taken in $\mathbb R^n$, the above, along with the definition in the note (b) in Section 1, implies that the exact controllability for (\ref{od-1}) over $[0,T]$ can be realized at $\{\tau_k\}_{k=1}^p$. This proves the sufficiency.

Finally, we show the necessity. Assume that the exact controllability for (\ref{od-1}) over $[0,T]$ can be realized at $\{\tau_k\}_{k=1}^p$. Then, by the definition in the note (b) in Section 1 and (\ref{od-1-1}), we get that for each $z_1 \in \mathbb R^n$, there exists $\{u_k\}_{k=1}^p \subset \mathbb R^m$ so that
\begin{eqnarray}\label{ode-2}
z_1&=&z(T; 0, \{\tau_k\}_{k=1}^p, \{u_k\}_{k=1}^p)
= \sum_{1\leq k\leq p} e^{A(T-\tau_k)}Bu_k\\
\nonumber
&\in& \mbox{Range}\;(e^{A(T-\tau_1)}B, \ldots, e^{A(T-\tau_n)}B).
\end{eqnarray}
This indicates that
\begin{equation*}
  \mathbb R^n \subset \mbox{Range}\;(e^{A(T-\tau_1)}B, \ldots, e^{A(T-\tau_n)}B).
\end{equation*}
Thus, (\ref{wanggsgs2.2}) is true, which leads to the necessity.

In summary, we end the proof of this theorem.
\end{proof}

\section{Unique continuation for system of heat equations}

Some connections among  $\{e^{\mathcal At}\}_{t\geq 0}$,
$\{e^{-At}\}_{t\geq 0}$ and $\{e^{\mathbf{\Delta} t}\}_{t\geq 0}$ (which denotes the $C_0$-semigroup generated by $\mathbf{\Delta}$ over $L^2(\Omega; \mathbb R^n)$) are given in the next Proposition~\ref{lm-5}.

\begin{proposition}\label{lm-5}
The following two equalities hold for all $t\geq 0$:
\begin{equation}\label{lm-5-1}
   e^{\mathcal A^* t} = e^{-A^*t}e^{\mathbf{\Delta} t}
   \;\;\mbox{and}\;\;
   e^{\mathcal A t} = e^{\mathbf{\Delta} t}e^{-At}.
\end{equation}
\end{proposition}

\begin{proof}
 Arbitrarily fix $\mathbf{z} \in L^2(\Omega; \mathbb R^n)$. One can easily  check that  the function $t \mapsto  e^{\mathbf{\Delta} t}\mathbf{z}$, $t > 0$, belongs to the following space:
\begin{equation}\label{lm-5-5}
   C^1\big((0,\infty);L^2(\Omega; \mathbb R^n)\big) \cap C\big((0, \infty); H^2(\Omega; \mathbb R^n)\cap H^1_0(\Omega; \mathbb R^n)\big).
\end{equation}
Define
\begin{equation}\label{lm-5-3}
  \mathbf{h}_\mathbf{z}(t) \triangleq e^{-A^*t}e^{\mathbf{\Delta} t}\mathbf{z}, \;\; t \geq 0.
\end{equation}
By (\ref{lm-5-3}) and (\ref{lm-5-5}), we get that
\begin{equation}\label{lm-5-2}
  \mathbf{h}_\mathbf{z} \in  C^1\big((0,\infty);L^2(\Omega; \mathbb R^n)\big) \cap C\big((0, \infty); H^2(\Omega; \mathbb R^n)\cap H^1_0(\Omega; \mathbb R^n)\big).
\end{equation}
Since  $\mathcal A^*=\mathbf{\Delta}-A^* $ and because $\mathbf{\Delta}$  is commutative with the operator $A^*$, it follows from (\ref{lm-5-3}) that for each $t>0$,
\begin{eqnarray}\label{lm-5-4}
\frac{d}{dt} \mathbf{h}_\mathbf{z} - \mathcal A^*\mathbf{h}_\mathbf{z}
&=&\frac{d}{dt} \mathbf{h}_\mathbf{z} - \mathbf{\Delta} \mathbf{h}_\mathbf{z} + A^*\mathbf{h}_\mathbf{z}\nonumber\\
&=&\Big[-A^*e^{-A^*t}e^{\mathbf{\Delta} t}\mathbf{z} + e^{-A^*t}\frac{d}{dt}e^{\mathbf{\Delta} t}\mathbf{z}\Big]
- \mathbf{\Delta} e^{-A^*t}e^{\mathbf{\Delta} t}\mathbf{z} + A^*e^{-A^*t}e^{\mathbf{\Delta} t}\mathbf{z}
\nonumber\\
 &=& e^{-A^*t}\Big[\frac{d}{dt}e^{\mathbf{\Delta} t}\mathbf{z} - \mathbf{\Delta} e^{\mathbf{\Delta} t}\mathbf{z}\Big] =0.
\end{eqnarray}
Meanwhile, we observe from (\ref{lm-5-3}) that
$\mathbf{h}_\mathbf{z}(0) = \mathbf{z}$.
This, along with (\ref{lm-5-4}), yields that
\begin{equation*}
  \mathbf{h}_\mathbf{z}(t) = e^{\mathcal A^*t}\mathbf{z}, \;\; t \geq 0,
\end{equation*}
which, together with  (\ref{lm-5-3}), leads to the first equality in (\ref{lm-5-1}).

 By taking the adjoint on  both sides of   the first equality in (\ref{lm-5-1}), we obtain the second equality in (\ref{lm-5-1}). This ends the proof of Proposition~\ref{lm-5}.
\end{proof}

The next Proposition \ref{lm-6} presents unique continuation property  for the
 system of heat equations. The key to proving this proposition is the use of the unique continuation
 property at one point in time for the heat equation.
  This property says that if a solution $y$ of the heat equation with the homogeneous Dirichelt
  boundary condition satisfies that for some $\tau>0$, $y(x,\tau)=0$ for a.e. $x\in \omega$,
  then $y\equiv 0$. (See, for instance,  \cite{Lin}, \cite{PW}  and \cite{PWZ}).

\begin{proposition}\label{lm-6}
Let $\omega_1$ be an open and nonempty subset of $\Omega$ and let $V$ be a subspace of $\mathbb R^n$.   Then the following two conclusions are true:

(i)  If $\mathbf{z} \in L^2(\Omega; \mathbb R^n)$
 and $\tau>0$, then
\begin{equation*}
  \mathbf{z}(x) \in V\;\; \mbox{for a.e.}\;\; x\in \Omega\Longleftrightarrow
  e^{\mathbf{\Delta} \tau}\mathbf{z}(x) \in V \;\; \mbox{for all}\;\; x \in \omega_1.
\end{equation*}

 (ii) If $\mathbf{z} \in L^2(\Omega; \mathbb R^n)$
 and $\tau>0$, then
  $\chi_{\omega_1}e^{\mathcal A \tau}\mathbf{z} = 0$ if and only if $\mathbf{z}=0$.
\end{proposition}

\begin{proof}
(i) Arbitrarily fix $\mathbf{z} \in L^2(\Omega; \mathbb R^n)$ and $\tau>0$.
 Let
 $\Psi(t)\triangleq e^{\mathbf{\Delta} t}\mathbf{z}$ for each $t\geq 0$.
 Then $\Psi$ satisfies that
 \begin{eqnarray}\label{lm-6-1}
\left\{\begin{array}{lll}
         \partial_t \Psi - \mathbf{\Delta} \Psi= 0  & \mbox{in} &  \mathbb R^+ \times \Omega,\\
         \Psi = 0 & \mbox{on} &  \mathbb R^+ \times \partial\Omega,\\
         \Psi(0) = \mathbf{z} & \mbox{in} &\Omega.
       \end{array}
\right.
\end{eqnarray}

We organize the rest of the proof by the following two steps:

  \textit{Step 1. We show that if  $\mathbf{z}(x) \in V$ for a.e. $x\in \Omega$, then
  $e^{\mathbf{\Delta} \tau}\mathbf{z}(x) \in V $ for all $x \in \omega_1$
  }

   Suppose that
 \begin{equation}\label{lm-6-2}
   \mathbf{z}(x) \in V \;\; \mbox{for a.e.}\;\; x\in \Omega.
 \end{equation}
 We aim to prove that
 \begin{equation}\label{lm-6-17-1}
   e^{\mathbf{\Delta} \tau}\mathbf{z}(x) \in V \;\; \mbox{for all}\;\; x \in \omega_1.
 \end{equation}
   Arbitrarily fix $\mathbf{\alpha}$ in $V^\bot$ which is  the orthogonal complement space of $V$ in $\mathbb R^n$.  Then  define a function
 $\psi_\alpha: \Omega\times [0,\infty)\rightarrow \mathbb{R}$
 in the following manner:
 \begin{equation}\label{lm-6-3}
  \psi_\alpha(x,t)\triangleq \langle e^{\mathbf{\Delta} t}\mathbf{z}(x), \alpha\rangle_{\mathbb R^n}, \;\; (x,t)\in \Omega \times [0, \infty).
 \end{equation}
By (\ref{lm-6-3}) and (\ref{lm-6-1}), one can directly check that  $\psi_\alpha$ satisfies  that
\begin{eqnarray}\label{lm-6-4}
\left\{\begin{array}{lll}
         \partial_t \psi_\alpha - \Delta \psi_\alpha = 0  & \mbox{in} & \Omega \times \mathbb R^+,\\
         \psi_\alpha = 0 & \mbox{on} &\partial\Omega \times \mathbb R^+.
       \end{array}
\right.
\end{eqnarray}
Meanwhile, since $\mathbf{\alpha} \in V^\bot$, it follows from   (\ref{lm-6-3}) and (\ref{lm-6-2}) that $\psi_\alpha(x,0)=0$ for a.e. $x\in\Omega$. This, along with  (\ref{lm-6-4}), yields that $\psi_\alpha\equiv 0$ over $\Omega\times(0,\infty)$,
which, together with  (\ref{lm-6-3}), indicates that
\begin{equation*}
\langle e^{\mathbf{\Delta} \tau}\mathbf{z}(x), \alpha\rangle_{\mathbb R^n}=0\;\;\mbox{for all}\;\;x\in\Omega.
\end{equation*}
Since $\alpha$ was arbitrarily taken from $V^\bot$, the above leads to (\ref{lm-6-17-1}).

  \textit{Step 2. We show that if $e^{\mathbf{\Delta} \tau}\mathbf{z}(x) \in V \;\; \mbox{for each}\;\; x \in \omega_1$, then
  $\mathbf{z}(x) \in V$ for a.e. $x\in \Omega$}

We only need to consider the case that  $\mbox{dim}\, V < n$,
since when $\mbox{dim}\,V = n$, the desired result is clearly true.
Suppose  that
 \begin{equation}\label{lm-6-6}
   e^{\mathbf{\Delta} \tau}\mathbf{z}(x) \in V \;\; \mbox{for all}\;\; x\in \omega_1.
 \end{equation}
 Arbitrarily take $\alpha \in V^\bot$. Let $\psi_\alpha$ be the function defined by (\ref{lm-6-3}). Then from  (\ref{lm-6-3}) and (\ref{lm-6-6}), we see that
 \begin{equation}\label{wang2.23}
 \psi_\alpha(x, \tau) = 0\;\; \mbox{for all}\;\; x\in\omega_1.
 \end{equation}
Since $\psi_\alpha$ solves the adjoint heat equation with the zero Dirichlet boundary condition (see (\ref{lm-6-4})), by (\ref{lm-6-4}) and (\ref{wang2.23}), using the unique continuation property for  heat equations (see, for instance,   \cite{Lin}, \cite{PW} and \cite{PWZ}), we see that
$\psi_\alpha(x, 0)=0$ for a.e. $x\in\Omega$.
This, along with (\ref{lm-6-3}), yields that
\begin{equation*}
  \langle \mathbf{z}(x), \alpha\rangle_{\mathbb R^n} = 0\;\;\mbox{for a.e.}\;\; x\in\Omega.
\end{equation*}
Since  $\alpha$ was arbitrarily taken from the finitely dimensional subspace $V^\bot$,
the above leads to that $\mathbf{z}(x)\in V$ for a.e. $x\in\Omega$.

By the results in Step 1 and Step 2, we see that the conclusion (i) of this proposition is true.

\vskip 5pt
(ii) Arbitrarily fix $\mathbf{z} \in L^2(\Omega; \mathbb R^n)$ and $\tau>0$.
It is clear that  $\chi_{\omega_1} e^{\mathcal A\tau} \textbf{z}=0$, when $\textbf{z}=0$. To show the reverse, we suppose that
\begin{eqnarray}\label{pro-2-ii-1}
 \chi_{\omega_1} e^{\mathcal A\tau} \textbf{z}=0.
\end{eqnarray}
 From  (\ref{pro-2-ii-1}) and the second equality in Proposition \ref{lm-5}, we see that
\begin{eqnarray}\label{wang2.25}
\chi_{\omega_1} e^{\mathbf{\Delta} \tau} (e^{-A \tau}\textbf{z})=\chi_{\omega_1} e^{\mathcal A\tau} \textbf{z}=0.
\end{eqnarray}
By (\ref{wang2.25}), we can use the conclusion (i) (in this proposition), with $V=\{0\}$,
to get that  $e^{-A \tau}\textbf{z}(x)=0$  a.e. $x\in\Omega$. This yields that
$\textbf{z}(x)=0$ for a.e. $x\in\Omega$, since  $e^{-A \tau}$ is invertible.
 Hence, the conclusion (ii) is true.

In summary, we end the proof of Proposition~\ref{lm-6}.
\end{proof}

\bigskip

\section{Proof of Theorem \ref{th-1}}

The following formula will be frequently used in the rest of this paper:
 For each $T>0$, $\mathbf{y}_0 \in L^2(\Omega; \mathbb R^n)$, $p\in\mathbb N^+$, $\{\tau_k\}_{k=1}^p \subset (0,T)$ (with $\tau_1<\cdots<\tau_p$),
$\{\mathbf{u}_k\}_{k=1}^p \subset L^2(\Omega; \mathbb R^m)$ and $\mathbf{z }\in L^2(\Omega; \mathbb R^n)$,  it stands that
\begin{eqnarray}\label{lm-1-1}
  \big\langle \mathbf{y}(T; \mathbf{y}_0, \{\tau_k\}_{k=1}^p, \{\mathbf{u}_k\}_{k=1}^p), \mathbf{z} \big\rangle
   &=& \langle \mathbf{y}_0, e^{\mathcal A^*T} \mathbf{z}   \rangle\\
\nonumber
  & &+  \sum_{k=1}^p  \big\langle  \mathbf{u}_k, \chi_{\omega}B^*   e^{\mathcal A^*(T-\tau_k)} \mathbf{z} \big\rangle_{L^2(\Omega; \mathbb R^m)}.
\end{eqnarray}
The equality (\ref{lm-1-1})  follows directly from  (\ref{solution-formula}).
Before proving Theorem \ref{th-1}, we present three  lemmas.
 The first two lemmas will be used in the proof of Theorem \ref{th-1}, while the  last one will be used in the proofs of both Theorem \ref{th-1} and  Theorem \ref{th-2}.

\begin{lemma}\label{lm-8}
There exists a linear map $C$ from $\mathbb R^n$ to $\mathbb R^m$ so that
\begin{equation}\label{lm-8-1}
  B C \alpha = \alpha\;\;\mbox{for each}\;\; \alpha \in \mbox{Range}\,B.
\end{equation}
\end{lemma}

\begin{proof}
When $B=0$,  (\ref{lm-8-1}) is trivial. Thus, we can assume, without loss of generality, that  $B\neq 0$.  Take a basis $\{\alpha_1, \ldots, \alpha_q\}$ of $\mbox{Range}\,B$, where $q \triangleq \mbox{dim}\,\mbox{Range}\,B$ ($q\geq 1$, since $B\neq 0$). Then for each $k\in\{1,\ldots,q\}$, there exists $\beta_k \in \mathbb R^m$ so that
\begin{equation}\label{lm-8-2}
  B\beta_k = \alpha_k.
\end{equation}
Define a linear map $\hat C$ from $\mbox{Range}\,B$ to $\mathbb R^m$ so that
\begin{equation}\label{lm-8-3}
  \hat C(\alpha_k) = \beta_k\;\;\mbox{for each}\;\;k\in\{1,\ldots,q\}.
\end{equation}
We claim that
\begin{equation}\label{lm-8-4}
  B\hat C(\alpha) = \alpha\;\;\mbox{for each}\;\;\alpha\in\mbox{Range}\,B.
\end{equation}
To this end, we arbitrarily fix $\alpha\in\mbox{Range}\,B$ and write
\begin{equation}\label{lm-8-5}
  \alpha = \sum_{k=1}^q a_k\alpha_k\;\; \mbox{with}\;\;\{a_k\}_{k=1}^q\subset\mathbb R.
\end{equation}
It follows from (\ref{lm-8-5}), (\ref{lm-8-3}) and (\ref{lm-8-2}) that
\begin{equation*}
  B\hat C(\alpha) = \sum_{k=1}^q a_kB\hat C(\alpha_k) = \sum_{k=1}^q a_kB\beta_k = \sum_{k=1}^q a_k\alpha_k = \alpha.
\end{equation*}
Since $\alpha$ was arbitrarily taken from $\mbox{Range}\,B$, the above leads to (\ref{lm-8-4}).

Finally, we observe that the map $\hat C$ is defined over $\mbox{Range}\,B$ (see (\ref{lm-8-3})). It is easy to extend $\hat C$ to be a linear operator over $\mathbb R^n$, with the property (\ref{lm-8-4}).
This ends the proof of Lemma~\ref{lm-8}.
\end{proof}

\begin{lemma}\label{lm-9}
Let $V_1,\dots,V_q$, with $q\in \mathbb{N}^+$, be  subspaces of $\mathbb R^n$ so that
\begin{equation}\label{lm-9-1}
  \sum_{k=1}^q V_k = \mathbb R^n.
\end{equation}
Then there are linear operators $P_1,\dots,P_q$ (from $\mathbb R^n$ to $\mathbb R^n$) so that
for each $k\in\{1,\ldots,q\}$, $P_k$ maps $\mathbb R^n$ into $V_k$ and so that
\begin{equation}\label{lm-9-2}
  \alpha = \sum_{k=1}^q P_k\alpha \;\;\mbox{for each}\;\; \alpha \in \mathbb R^n.
\end{equation}
\end{lemma}

\begin{proof}
Take a basis $\{\alpha_1, \ldots, \alpha_n\}$ of $\mathbb R^n$. By (\ref{lm-9-1}), we find that for each $j\in\{1,\ldots,n\}$, there exits $(\alpha_{j,1}, \alpha_{j,2}, \ldots, \alpha_{j,q}) \in \Pi_{k=1}^q V_k$ so that
\begin{equation}\label{lm-9-3}
  \alpha_j = \sum_{k=1}^q\alpha_{j,k}.
\end{equation}
For each $k\in\{1,\ldots,q\}$, we define a linear map $P_k: \mathbb R^n \rightarrow V_k$
in the following manner:
\begin{equation}\label{lm-9-4}
  P_k(\alpha_j) = \alpha_{j,k}\;\;\mbox{for each}\;\;j\in\{1,\ldots,n\}.
\end{equation}
We claim that the above $\{P_k\}_{k=1}^q$ satisfies (\ref{lm-9-2}). For this purpose, we arbitrarily take $\alpha \in \mathbb R^n$. Write
\begin{equation}\label{lm-9-5}
  \alpha = \sum_{j=1}^n a_j\alpha_j \;\;\mbox{with}\;\;\{a_j\}_{j=1}^n\subset \mathbb R.
\end{equation}
 Then it follows from (\ref{lm-9-5}), (\ref{lm-9-3}) and (\ref{lm-9-4}) that
\begin{eqnarray*}
\alpha &=& \sum_{j=1}^n a_j(\sum_{k=1}^q\alpha_{j,k})
= \sum_{k=1}^q( \sum_{j=1}^n a_j\alpha_{j,k})
= \sum_{k=1}^q(\sum_{j=1}^n a_jP_k(\alpha_j))
= \sum_{k=1}^qP_k(\sum_{j=1}^n a_j\alpha_j)\\
&=& \sum_{k=1}^q P_k(\alpha).
\end{eqnarray*}
Since $\alpha$ was arbitrarily taken from $\mathbb R^n$, the above leads to (\ref{lm-9-2}). We end the proof of Lemma \ref{lm-9}.
\end{proof}

\begin{lemma}\label{lemmawang3.3}
Let $(A,B)\in \mathbb{R}^{n\times n}\times\mathbb{R}^{n\times m}$ satisfies that
\begin{equation}\label{wanggs3.12}
  \mbox{rank}\;(B, AB, A^2B, \ldots, A^{n-1}B) < n.
\end{equation}
Then there is  $\hat{\mathbf{z}} \in L^2(\Omega; \mathbb R^n)\setminus\{0\}$ so that
\begin{equation}\label{wanggs3.13}
  \chi_{\omega}B^*e^{\mathcal A^*t}\hat {\mathbf{z}} = 0\;\;\mbox{for each}\;\; t >0.
\end{equation}
\end{lemma}
\begin{proof}
We first  claim that
\begin{equation}\label{lm-14-10}
  \bigcap_{k=0}^{n-1}\mbox{ker\,}(B^*(A^*)^k) \neq \{0\}.
\end{equation}
Indeed, since  $(A, B)$ satisfies (\ref{wanggs3.12}),  there is $\alpha \in \mathbb R^n \setminus \{0\}$ so that for all $k\in\{0,\ldots,n-1\}$,
\begin{equation*}
  \alpha^\intercal(B, AB, \ldots, A^{n-1}B)=0,\;\; i.e.,\;\; \alpha^\intercal A^{k}B = 0,
\end{equation*}
from which, it follows  that
\begin{equation*}
  B^*(A^*)^{k}\alpha = (\alpha^\intercal A^{k}B)^\intercal=0\;\; \mbox{for all}\;\; k\in\{0,\ldots,n-1\}.
\end{equation*}
This implies that
\begin{equation*}
  \alpha \in \mbox{ker\,}(B^*(A^*)^k)\;\; \mbox{for all}\;\;k\in\{0,\ldots,n-1\}.
\end{equation*}
Since $\alpha\neq 0$, the above  leads to (\ref{lm-14-10}).

Next, we  let
\begin{equation}\label{lm-14-12}
  V \triangleq \bigcap_{k=0}^{\infty}\mbox{ker}\,(B^*(A^*)^k).
\end{equation}
By   the Cayley-Hamilton theorem, we have that $V=\bigcap_{k=0}^{n-1}\mbox{ker}\,(B^*(A^*)^k)$. Then by (\ref{lm-14-10}),
there is $\alpha\in \mathbb{R}^n$ so that
\begin{equation}\label{lm-14-10-1}
  \alpha \in V\setminus\{0\}.
\end{equation}
We define  a function $\mathbf{z}_{\alpha} : \Omega \rightarrow \mathbb R^n$ in the following manner:
\begin{equation}\label{lm-14-11}
  \mathbf{z}_\alpha(x) = \alpha \;\; a.e. \;\; x\in\Omega.
\end{equation}
Then by Proposition \ref{lm-5}  (see the first equality in (\ref{lm-5-1})), we find that for each $t>0$,
\begin{eqnarray}\label{lm-14-13}
\chi_{\omega}B^*e^{\mathcal A^*t}\mathbf{z}_\alpha &=& \chi_{\omega}B^*e^{-A^*t}e^{\mathbf{\Delta} t}\mathbf{z}_\alpha
= B^*e^{-A^*t}\chi_{\omega}e^{\mathbf{\Delta} t}\mathbf{z}_\alpha
\nonumber\\
&=& B^*\Big(\sum_{k\geq0}\frac{1}{k!}(-t)^k (A^*)^k\Big)\chi_{\omega}e^{\mathbf{\Delta} t}\mathbf{z}_\alpha
\nonumber\\
&=& \sum_{k\geq0}\frac{1}{k!}(-t)^k B^*(A^*)^k\chi_{\omega}e^{\mathbf{\Delta} t}\mathbf{z}_\alpha.
\end{eqnarray}
Meanwhile, since $\mathbf{z}_\alpha(x)\in V\;\mbox{a.e.}\; x\in\Omega$ (see (\ref{lm-14-10-1}) and (\ref{lm-14-11})), it follows by (i) of Proposition \ref{lm-6} (where $\omega_1 = \omega$, $\tau=t$ and $\mathbf{z} = \mathbf{z}_{\alpha}$) that for each $t>0$,
\begin{equation*}
  \chi_{\omega}(x)e^{\mathbf{\Delta} t}\mathbf{z}_\alpha(x) \in V\;\;\mbox{for a.e.}\;\; x\in\Omega,
\end{equation*}
which, together with (\ref{lm-14-12}), yields that for each $t>0$,
\begin{equation*}
 B^*(A^*)^k\chi_{\omega}(x)e^{\mathbf{\Delta} t}\mathbf{z}_\alpha(x)=0\;\;\mbox{for a.e.}\;\; x\in\Omega.
\end{equation*}
The above, along with  (\ref{lm-14-13}), leads to (\ref{wanggs3.13}) with $\hat {\mathbf{z}} = \mathbf{z}_\alpha$. This ends the proof.
\end{proof}

We now on the position to prove Theorem~\ref{th-1}.

\begin{proof}[Proof of Theorem \ref{th-1}]
(i) Assume that
\begin{equation}\label{th-1-2}
  \Omega \setminus \overline{\omega} \neq \emptyset.
\end{equation}
By contradiction, we suppose that the system (\ref{heat-eq-ip}) were null controllable over $[0, \hat T]$ for some $\hat T>0$. Then, according to  (i) of Definition \ref{Def-2}, for an arbitrarily fixed
\begin{equation}\label{th-1-3}
  \hat{\mathbf{y}}_0 \in L^2(\Omega; \mathbb R^n)\setminus\{0\},
\end{equation}
  there is $\hat p \in \mathbb N^+$, $\{\hat\tau_k\}_{k=1}^{\hat p} \subset (0,\hat T)$ (with $\hat\tau_1<\cdots<\hat\tau_p$),
and $\{\hat{\mathbf{u}}_k\}_{k=1}^{\hat p} \subset L^2(\Omega; \mathbb R^m)$
so that
\begin{equation*}
  \mathbf{y}(\hat T; \hat{\mathbf{y}}_0, \{\hat\tau_k\}_{k=1}^{\hat p}, \{\hat{\mathbf{u}}_k\}_{k=1}^{\hat p}) =0.
\end{equation*}
This, along with (\ref{solution-formula}), indicates that
\begin{equation*}
  0 = \mathbf{y}(\hat T; \hat{\mathbf{y}}_0, \{\hat\tau_k\}_{k=1}^{\hat p}, \{\hat{\mathbf{u}}_k\}_{k=1}^{\hat p})
   = e^{\mathcal A(T-\hat\tau_{\hat p})}\Big( e^{\mathcal A\hat\tau_{\hat p}}\hat{\mathbf{y}}_0 + \sum_{k=1}^{\hat p} e^{\mathcal A(\hat\tau_{\hat p}-\hat\tau_k)}\chi_{\omega}B\hat{\mathbf{u}}_k\Big).
\end{equation*}
From this and (ii) of Proposition \ref{lm-6}, where
\begin{equation*}
  \omega_1= \Omega, \tau= T-\hat\tau_{\hat p} \;\;\mbox{and}\;\;
\mathbf{z} =  e^{\mathcal A\hat\tau_{\hat p}}\hat{\mathbf{y}}_0 + \sum_{k=1}^{\hat p} e^{\mathcal A(\hat\tau_{\hat p}-\hat\tau_k)}\chi_{\omega}B\hat{\mathbf{u}}_k,
\end{equation*}
we  find that
\begin{eqnarray}\label{th-1-5}
 0
 &=& e^{\mathcal A\hat\tau_{\hat p}}\hat{\mathbf{y}}_0 + \sum_{k=1}^{\hat p} e^{\mathcal A(\hat\tau_{\hat p}-\hat\tau_k)}\chi_{\omega}B\hat{\mathbf{u}}_k
 \nonumber\\
 &=& e^{\mathcal A(\hat\tau_{\hat p} - \hat\tau_{\hat p-1})}\Big( e^{\mathcal A\hat\tau_{\hat p-1}}\hat{\mathbf{y}}_0 + \sum_{k=1}^{\hat p-1} e^{\mathcal A(\hat\tau_{\hat p-1}-\hat\tau_k)}\chi_{\omega}B\hat{\mathbf{u}}_k\Big) + \chi_{\omega}B\hat u_p.
\end{eqnarray}
Meanwhile, from (\ref{th-1-2}), there exists $B_r(\mathbf{x}_0)\subset\Omega$ so that
\begin{equation}\label{th-1-6}
  B_r(\mathbf{x}_0) \subset \Omega \setminus \overline{\omega},
\end{equation}
where $B_r(\mathbf{x}_0)$ denotes the ball in $\mathbb R^{N}$, centred at $\mathbf{x}_0$ and of radius $r>0$. From (\ref{th-1-6}), it follows that $\chi_{\omega}B\hat u_p=0\;\;\mbox{on}\; B_r(\mathbf{x}_0)$, which, along with  (\ref{th-1-5}), yields that
\begin{equation*}
  e^{\mathcal A(\hat\tau_{\hat p} - \hat\tau_{\hat p-1})}\Big( e^{\mathcal A\hat\tau_{\hat p-1}}\hat{\mathbf{y}}_0 + \sum_{k=1}^{\hat p-1} e^{\mathcal A(\hat\tau_{\hat p-1}-\hat\tau_k)}\chi_{\omega}B\hat{\mathbf{u}}_k\Big)
  = 0 \;\; \mbox{on}\; B_r(\mathbf{x}_0).
\end{equation*}
From this and  (ii) of Proposition \ref{lm-6}, where
\begin{equation*}
  \omega_1= B_r(\mathbf{x}_0), \tau= \hat\tau_{\hat p} - \hat\tau_{\hat p-1} \;\;\mbox{and}\;\;
\mathbf{z} =  e^{\mathcal A\hat\tau_{\hat p-1}}\hat{\mathbf{y}}_0 + \sum_{k=1}^{\hat p-1} e^{\mathcal A(\hat\tau_{\hat p-1}-\hat\tau_k)}\chi_{\omega}B\hat{\mathbf{u}}_k,
\end{equation*}
we see that
\begin{equation}\label{wang3.16}
  e^{\mathcal A\hat\tau_{\hat p-1}}\hat{\mathbf{y}}_0 + \sum_{k=1}^{\hat p-1} e^{\mathcal A(\hat\tau_{\hat p-1}-\hat\tau_k)}\chi_{\omega}B\hat{\mathbf{u}}_k = 0.
\end{equation}
Following the same way as that showing (\ref{wang3.16}),  one can verify that
when $1\leq q \leq\hat p-1$,
\begin{equation*}
   e^{\mathcal A\hat\tau_{q}}\hat{\mathbf{y}}_0 + \sum_{k=1}^{q} e^{\mathcal A(\hat\tau_{q}-\hat\tau_k)}\chi_{\omega}B\hat{\mathbf{u}}_k = 0.
\end{equation*}
In particular, by taking $q=1$ in the above, we obtain that
\begin{equation*}
  e^{\mathcal A\hat\tau_1}\hat{\mathbf{y}}_0 + \chi_{\omega}B\hat u_1 = 0.
\end{equation*}
This, combined with  (\ref{th-1-6}), yields that
$ e^{\mathcal A\hat\tau_1}\hat{\mathbf{y}}_0 =0$ over $B_r(\mathbf{x}_0)$.
From the above and  (ii) of Proposition \ref{lm-6}, where
$\omega_1= B_r(\mathbf{x}_0)$, $\tau= \hat\tau_1$ and $\mathbf{z}=\hat{\mathbf{y}}_0$,
we find that $\hat{\mathbf{y}}_0 =0$,
which contradicts (\ref{th-1-3}). Hence, when $\omega$ satisfies (\ref{th-1-2}),  the system (\ref{heat-eq-ip}) is not null controllable on $[0,T]$ for any $T>0$. This ends the proof of the conclusion (i) of Theorem \ref{th-1}.

(ii) Suppose that
\begin{equation}\label{th-1-7}
  \omega = \Omega.
\end{equation}

{\it Step 1. We  prove the sufficiency.}

 Assume that $(A,B)$ satisfies Kalman's controllability rank condition (\ref{th-1-1}).
We aim to show the null controllability for the system (\ref{heat-eq-ip}).
To this end, we arbitrarily fix $T>0$ and  an increasing sequence $\{\tau_k\}_{k=1}^n \subset (0,T)$, with $\tau_n - \tau_1 < d_A$ (given by (\ref{th-o-1-1})). Then by Theorem \ref{th-o-1} and (\ref{th-1-1}), we get that
\begin{equation}\label{th-1-9}
 \mbox{rank}\;(e^{A\tau_1}B, \ldots, e^{A\tau_{n}}B) = n.
\end{equation}
For each $k \in \{1,\ldots,n\}$, we define a subspace $V_k$ of $\mathbb R^n$ in the following manner:
\begin{equation}\label{th-1-10}
  V_k = \{e^{A\tau_k}B\alpha \in \mathbb R^n: \;\; \alpha \in \mathbb R^m\}.
\end{equation}
From (\ref{th-1-9}) and (\ref{th-1-10}), we see that
$\mathbb R^n = \sum_{k=1}^n V_k$.
This, together with  Lemma \ref{lm-9}, yields that for each $k\in\{1,\ldots,n\}$, there is a linear map $P_{k}: \mathbb R^n \rightarrow V_k$ so that
\begin{equation}\label{th-1-11}
 \alpha = \sum_{k=1}^n P_{k}\alpha \;\;\mbox{for each}\;\; \alpha\in\mathbb R^n.
\end{equation}

We now claim that for an arbitrarily fixed $\mathbf{y}_0 \in L^2(\Omega; \mathbb R^n)$,
\begin{equation}\label{th-1-13}
   e^{\mathbf\Delta\tau_k}e^{-A\tau_k}P_{k}\mathbf{y}_0(x) \in \mbox{Range}\,B\;\; \mbox{ a.e.}\;\; x\in\Omega, \;\;\mbox{when}\;\;1\leq k\leq n.
\end{equation}
For this purpose, we observe from  (\ref{th-1-10}) that $e^{-A\tau_k}V_k \subset \mbox{Range}\,B$ for each  $k\in\{1,\ldots,n\}$.
This yields that  for each $k\in\{1,\ldots,n\}$,
\begin{equation}\label{wang3.22}
  e^{-A\tau_k}P_{k}\mathbf{y}_0(x) \in \mbox{Range}\,B\;\; \mbox{for a.e.}\;\; x\in\Omega.
\end{equation}
By (\ref{wang3.22}) and (i) of Proposition \ref{lm-6}, where
\begin{equation*}
  \omega_1= \Omega, V=\mbox{Range}\,B, \tau=\tau_k \;\;\mbox{and}\;\;
\mathbf{z} = e^{-A\tau_k}P_{k}\mathbf{y}_0,
\end{equation*}
  we are led to (\ref{th-1-13}).

Next,  according to Lemma \ref{lm-8},  there exists a linear map $C$ from $\mathbb R^n$ to $\mathbb R^m$ so that
\begin{equation}\label{th-1-12}
  BC\alpha = \alpha\;\; \mbox{for all}\;\; \alpha\in\mbox{Range}\,B.
\end{equation}
From (\ref{th-1-13}) and (\ref{th-1-12}), we see that for each  $k\in\{1,\ldots,n\}$,
\begin{equation}\label{th-1-14}
  e^{\mathbf{\Delta}\tau_k}e^{-A\tau_k}P_{k}\mathbf{y}_0 + B\hat{\mathbf{u}}_k = 0,
\end{equation}
where the control $\hat{\mathbf{u}}_k$ is defined by
\begin{equation*}
  \hat{\mathbf{u}}_k(x) \triangleq -Ce^{\mathbf{\Delta}\tau_k}e^{-A\tau_k}P_{k}\mathbf{y}_0(x),\;\; x\in\Omega.
\end{equation*}
Since $\omega = \Omega$ (see (\ref{th-1-7})), it follows from (\ref{solution-formula}), (\ref{th-1-11}), Proposition \ref{lm-5}
(see the second equality in (\ref{lm-5-1})) and (\ref{th-1-14}) that
\begin{eqnarray*}
\mathbf{y}(T; \mathbf{y}_0, \{\tau_k\}_{k=1}^n, \{\hat{\mathbf{u}}_k\}_{k=1}^n)
&=& e^{\mathcal AT}\mathbf{y}_0 + \sum_{k=1}^ne^{\mathcal A(T-\tau_k)}B\hat{\mathbf{u}}_k\\
&=& e^{\mathcal AT}\big(\sum_{k=1}^nP_{k}\mathbf{y}_0\big)+ \sum_{k=1}^ne^{\mathcal A(T-\tau_k)}B\hat{\mathbf{u}}_k\\
&=& \sum_{k=1}^ne^{\mathcal A(T-\tau_k)}\big(e^{\mathcal A\tau_k}P_{k}\mathbf{y}_0 + B\hat{\mathbf{u}}_k\big)\\
&=& \sum_{k=1}^ne^{\mathcal A(T-\tau_k)}\big(e^{\mathbf{\Delta}\tau_k}e^{-A\tau_k}P_{k}\mathbf{y}_0 + B\hat{\mathbf{u}}_k\big)
=0.
\end{eqnarray*}
Because $T>0$ and $\mathbf{y}_0 \in L^2(\Omega; \mathbb R^n)$ were arbitrarily taken, the above, along with (ii) of Definition \ref{Def-2}, leads to the null controllability for the system (\ref{heat-eq-ip}).

{\it Step 2. We show  the necessity.}

 Assume that the system (\ref{heat-eq-ip}) is null controllable. We aim to prove that the pair $(A,B)$ satisfies Kalman's controllability rank condition (\ref{th-1-1}). By contradiction, we suppose that
\begin{equation}\label{lm-14-9}
  \mbox{rank}\;(B, AB, A^2B, \ldots, A^{n-1}B) < n.
\end{equation}
Then
by Lemma \ref{lemmawang3.3},
 there would be  $\hat{\mathbf{z}} \in L^2(\Omega; \mathbb R^n)\setminus\{0\}$ so that
\begin{equation}\label{lm-14-11-11}
  \chi_{\omega}B^*e^{\mathcal A^*t}\hat {\mathbf{z}} = 0\;\;\mbox{for each}\;\; t >0.
\end{equation}
 Since the system (\ref{heat-eq-ip}) is null controllable, it follows by (ii) of Definition \ref{Def-2} that the system
 (\ref{heat-eq-ip}) is null controllable over $[0,T]$ for each $T>0$. Then by this and (i) of Definition \ref{Def-2}, there is $p\in\mathbb N^+$, $\{\tau_{k}\}_{k=1}^{p} \subset(0,1)$ and $\{\mathbf{v}_{k}\}_{k=1}^{p} \subset L^2(\Omega; \mathbb R^m)$ so that
\begin{equation}\label{lm-14-14}
  0=\mathbf{y}(1; e^{\mathcal A^*}\hat{\mathbf{z}}, \{\tau_{k}\}_{k=1}^{p}, \{\mathbf{v}_{k}\}_{k=1}^{p})
  = e^{\mathcal A} (e^{\mathcal A^*}\hat{\mathbf{z}})+\mathbf{y}(1; 0, \{\tau_{k}\}_{k=1}^{p}, \{\mathbf{v}_{k}\}_{k=1}^{p}).
\end{equation}
By (\ref{lm-14-14}), (\ref{lm-1-1}) and (\ref{lm-14-11-11}), we get that
\begin{eqnarray*}
\langle e^{\mathcal A}(e^{\mathcal A^*}\hat{\mathbf{z}}), \hat{\mathbf{z}} \rangle
&=& -\big\langle\mathbf{y}(1; 0, \{\tau_{k}\}_{k=1}^{p}, \{\mathbf{v}_{k}\}_{k=1}^{p}),\hat {\mathbf{z}} \big\rangle
 \\
%&=&\Big\langle \sum_{k=1}^n e^{\mathcal A(T-\tau_k)}\chi_{\omega}Bv_k,\hat {\mathbf{z}} \Big\rangle - \big\langle\hat {\mathbf{z}} - \mathbf{y}(T; \mathbf{y}_0, \{\tau_{k,\varepsilon}\}_{k=1}^{p}, \{v_{k,\varepsilon}\}_{k=1}^{p}),\hat {\mathbf{z}} \big\rangle\\
&=&-\sum_{k=1}^{p} \Big\langle  \mathbf{v}_{k},\chi_{\omega}B^*e^{\mathcal A^*(1-\tau_{k})}\hat{\mathbf{z}} \Big\rangle
=0.
\end{eqnarray*}
This implies that $e^{\mathcal A^*}\hat {\mathbf{z}}=0$, which, combined with  Proposition \ref{lm-5}, shows that
\begin{eqnarray*}
 e^{\mathbf{\Delta}} \hat {\mathbf{z}}= e^{A^*} e^{-A^*} e^{\mathbf{\Delta}} \hat {\mathbf{z}}=e^{A^*}e^{\mathcal A^*}\hat {\mathbf{z}} = 0.
\end{eqnarray*}
From the above and  (i) of Proposition \ref{lm-6}, where
\begin{eqnarray*}
 \omega_1=\Omega,~V=\{0\},~\tau=1 \;\;\mbox{and}\;\;  {\mathbf{z}}=\hat {\mathbf{z}},
\end{eqnarray*}
we find that $\hat {\mathbf{z}}=0 $. This leads to a contradiction, since $\hat {\mathbf{z}}$  is not zero. Hence, $(A,B)$ satisfies Kalman's controllability rank condition (\ref{th-1-1}). This proves the necessity.

\vskip 5pt
In summary, we end the proof of Theorem \ref{th-1}.
\end{proof}

\bigskip

\section{The proof of Theorem \ref{th-2}}

The key  to proving Theorem \ref{th-2} is the following unique continuation property.
\begin{theorem}\label{lm-4}
 Let $T>0$ and $p\in \mathbb{N}^+$. Let $\{\tau_k\}_{k=1}^p\subset (0,T)$ be an increasing sequence. Then the following two statements are equivalent:

 (i) It holds that
\begin{equation*}
  rank\;(e^{A\tau_1}B, \ldots, e^{A\tau_{p}}B) = n.
\end{equation*}

(ii) If    $\mathbf{z} \in L^2(\Omega; \mathbb R^n)$, then
\begin{equation*}
  \chi_{\omega}B^*e^{\mathcal A^*(T-\tau_k)} \mathbf{z} = 0 \;\; \mbox{for all}\;\; k \in \{1,\ldots, p\}\Longrightarrow \mathbf{z}=0\;\;\mbox{over}\;\;\Omega.
\end{equation*}
\end{theorem}

\begin{proof} The proof is divided into the following two steps.\\
{\it Step 1. We show that (i) $\Longrightarrow$ (ii).}\\
      Suppose that (i) is true.
Let $\mathbf{z} \in L^2(\Omega; \mathbb R^n)$ satisfy that
\begin{equation}\label{lm-4-1-5}
  \chi_{\omega}B^*e^{\mathcal A^*(T-\tau_k)} \mathbf{z} = 0 \;\; \mbox{for all}\;\; k \in \{1,\ldots, p\}.
\end{equation}
It suffices  to show that $\mathbf{z}=0$.
For this purpose,  we first claim that
\begin{equation}\label{lm-4-1-2}
  \bigcap_{k=1}^p \mbox{ker}\;\left(B^*e^{-A^*(T-\tau_k)}\right) = \{0\}.
\end{equation}
Indeed, if $\alpha \in \bigcap_{k=1}^p \mbox{ker}\;\left(B^*e^{-A^*(T-\tau_k)}\right)$
(where $\alpha$ is a column vector in $\mathbb{R}^n$), then
\begin{equation*}
  B^*e^{-A^*(T-\tau_k)}\alpha = 0 \;\;\mbox{for all}\;\; k \in \{1,\ldots, p\},
\end{equation*}
from which, it follows that
\begin{eqnarray*}
\alpha^\intercal e^{-AT} (e^{A\tau_1}B, \cdots, e^{A\tau_p}B)
&=& (\alpha^\intercal e^{-A(T-\tau_1)}B, \cdots, \alpha^\intercal e^{-A(T-\tau_p)}B) \\
&=& \left(
                        \begin{array}{ccc}
                          B^*e^{-A^*(T-\tau_1)}\alpha \\
                          \vdots \\
                          B^*e^{-A^*(T-\tau_p)}\alpha\\
                        \end{array}
                      \right)^\intercal = 0.
\end{eqnarray*}
The above,  along with (i), yields that
$\alpha^\intercal e^{-AT}=0$, which leads to that  $\alpha =0$. Hence, (\ref{lm-4-1-2}) is true.

Next, by  Proposition \ref{lm-5} and (\ref{lm-4-1-5}), we get that for each $k \in \{1,\ldots, p\}$,
\begin{eqnarray*}
 B^*e^{-A^*(T-\tau_k)}\chi_{\omega}e^{\mathbf{\Delta}(T-\tau_k)}\mathbf{z}
 = \chi_{\omega}B^*e^{-A^*(T-\tau_k)}e^{\mathbf{\Delta}(T-\tau_k)} \mathbf{z}
 = \chi_{\omega}B^*e^{\mathcal A^*(T-\tau_k)} \mathbf{z} = 0.
\end{eqnarray*}
This indicates that for each $k \in \{1,\ldots, p\}$,
\begin{equation}\label{lm-4-3}
  e^{\mathbf{\Delta}(T-\tau_k)}\mathbf{z}(x) \in \mbox{ker}\; \left(B^*e^{-A^*(T-\tau_k)}\right)\;\; \mbox{for a.e.}\;\; x\in\omega.
\end{equation}
By (\ref{lm-4-3}), for each $k \in \{1,\ldots, p\}$, we can apply (i) of Proposition \ref{lm-6}, where
\begin{equation*}
V= \mbox{ker}\; \left(B^*e^{-A^*(T-\tau_k)}\right),\;\; \omega_1=\omega,\;\; \tau=T-\tau_k,
\end{equation*}
to obtain that
\begin{equation*}
  \mathbf{z}(x) \in \bigcap_{k=1}^p \mbox{ker}\;\left(B^*e^{-A^*(T-\tau_k)}\right)\;\; \mbox{for a.e.}\;\; x\in\Omega.
\end{equation*}
The above, together with (\ref{lm-4-1-2}), leads to that $\mathbf{z}=0$. Hence, (ii) holds.

{\it Step 2. We show that (ii) $\Longrightarrow$ (i).}

Suppose, by contradiction, that (ii) were true, but (i) did not hold. Then we would have that
\begin{equation*}
  rank\;(e^{A\tau_1}B, \ldots, e^{A\tau_{p}}B)<n.
\end{equation*}
Thus,  we can apply Lemma \ref{lemmawang3.3} to find
   $\hat{\mathbf{z}} \in L^2(\Omega; \mathbb R^n)\setminus\{0\}$ so that
$\chi_{\omega}B^*e^{\mathcal A^*t}\hat {\mathbf{z}}=0$ for each $t >0$. This, in particular, implies that
\begin{equation}\label{gnaw5.4}
\chi_{\omega}B^*e^{\mathcal A^*(T-\tau_k)}\hat {\mathbf{z}}=0\;\;\mbox{for all}\;\;k=1,\dots,p.
\end{equation}
From (\ref{gnaw5.4}) and (ii), we see that $\hat {\mathbf{z}}=0$, which leads to a contradiction. Hence, (ii) implies (i).

In summary, we end the proof
of Theorem~\ref{lm-4}.
\end{proof}

Now, we are on the position to prove Theorem \ref{th-2}.

\begin{proof}[Proof of Theorem \ref{th-2}]
(i) We divide the proof into the following two steps:

\textit{Step 1. We show the sufficiency.}

\noindent
Assume that the pair $(A, B)$ satisfies  Kalman's controllability rank condition (\ref{th-1-1}).
We aim to show the approximate controllability for the system (\ref{heat-eq-ip}).
 To this end, we arbitrarily fix  $T>0$ and  an increasing sequence $\{\hat\tau_k\}_{k=1}^n \subset (0,T)$ with $\hat\tau_n-\hat\tau_1<d_A$ (given by (\ref{th-o-1-1})). Then by
Theorem \ref{th-o-1}, we get that
\begin{equation}\label{lm-7-3}
  \mbox{rank}\;(e^{A\hat\tau_1}B, \ldots, e^{A\hat\tau_{n}}B) = n.
\end{equation}
Define an operator
$\mathcal{G}_T: (L^2(\Omega; \mathbb R^m))^n \rightarrow L^2(\Omega; \mathbb R^n)$ in the following manner:
\begin{equation}\label{lm-7-4}
  \mathcal{G}_T(\{\mathbf{u}_k\}_{k=1}^n) = \mathbf{y}(T; 0, \{\hat\tau_k\}_{k=1}^n, \{\mathbf{u}_k\}_{k=1}^n),\;\;
  \forall\, \{\mathbf{u}_k\}_{k=1}^n \in(L^2(\Omega; \mathbb R^m))^n.
\end{equation}
We claim that the range of  the map $\mathcal G$ is dense in $L^2(\Omega; \mathbb R^n)$, i.e.,
\begin{equation}\label{lm-7-5}
  \overline{\mbox{Range}\,\mathcal{G}_T} = L^2(\Omega; \mathbb R^n).
\end{equation}
By contradiction, we suppose that (\ref{lm-7-5}) were not true.
Then  there would be
\begin{equation}\label{lm-7-6}
  \mathbf{z} \in \overline{\mbox{Range}\,\mathcal{G}_T}^\bot \setminus \{0\}.
\end{equation}
By (\ref{lm-7-6}), (\ref{lm-7-4}) and (\ref{lm-1-1}), we find that for all $\{\mathbf{u}_k\}_{k=1}^n \subset L^2(\Omega; \mathbb R^n)$,
\begin{eqnarray*}
0 &=& \langle \mathbf{z}, \mathcal{G}_T(\{\mathbf{u}_k\}_{k=1}^n) \rangle
= \langle \mathbf{z}, \mathbf{y}(T; 0, \{\hat \tau_k\}_{k=1}^n, \{\mathbf{u}_k\}_{k=1}^n) \rangle\\
&=& \sum_{k=1}^n \langle \chi_{\omega}B^*e^{\mathcal A^*(T-\hat\tau_k)}\mathbf{z},  \mathbf{u}_k \rangle_{L^2(\Omega; \mathbb R^m)}.
\end{eqnarray*}
From the above, we see  that
\begin{equation*}
  \chi_{\omega}B^*e^{\mathcal A^*(T-\hat\tau_k)}\mathbf{z} = 0, \;\; \forall\, k \in \{1,\ldots, n\},
\end{equation*}
This, along with (\ref{lm-7-3}) and Theorem \ref{lm-4}, shows that
$\mathbf{z}=0$ over $\Omega$,
which contradicts (\ref{lm-7-6}). Hence (\ref{lm-7-5}) is true.

Now, we will use (\ref{lm-7-5}) to prove  the approximate controllability for the system (\ref{heat-eq-ip}) over $[0,T]$. For this purpose, we arbitrary take $\mathbf{y_0}, \mathbf{y_1} \in L^2(\Omega; \mathbb R^n)$ and $\varepsilon > 0$. Then by (\ref{lm-7-5}), we see that
\begin{equation*}
 \mathbf{y_1} - e^{\mathcal AT}\mathbf{y_0} \in \overline{\mbox{Range}\,\mathcal{G}_T}.
\end{equation*}
Thus, there exists $\{\hat {\mathbf{u}}_k\}_{k=1}^n \subset L^2(\Omega; \mathbb R^m)$ so that
\begin{eqnarray}\label{lm-7-8}
\|\mathcal{G}_T(\{\hat {\mathbf{u}}_k\}_{k=1}^n) - (\mathbf{y_1} - e^{\mathcal AT}\mathbf{y_0})\| \leq \varepsilon.
\end{eqnarray}
It follows from (\ref{lm-7-4}) and (\ref{lm-7-8}) that
\begin{eqnarray*}
\|\mathbf{y}(T; \mathbf{y_0}, \{\hat \tau_k\}_{k=1}^n, \{\hat {\mathbf{u}}_k\}_{k=1}^n) - \mathbf{y_1}\|
&=& \|\mathbf{y}(T; 0, \{\hat \tau_k\}_{k=1}^n, \{\hat {\mathbf{u}}_k\}_{k=1}^n) - (\mathbf{y_1}- e^{\mathcal AT}\mathbf{y_0})\|\\
&=& \|\mathcal{G}_T(\{\hat {\mathbf{u}}_k\}_{k=1}^n) - (\mathbf{y_1} - e^{\mathcal AT}\mathbf{y_0})\| \leq \varepsilon.
\end{eqnarray*}
This leads to the   approximate controllability for the system (\ref{heat-eq-ip}) over $[0,T]$ (see (i) of Definition \ref{Def-3}). Since $T>0$
 was arbitrarily fixed,  the approximate controllability for the system (\ref{heat-eq-ip}) over $[0,T]$
 follows at once (see (ii) of   Definition \ref{Def-3}).

 \vskip 5pt
 \textit{Step 2. We prove the necessity.}

 \noindent
 Assume that the system (\ref{heat-eq-ip}) has the approximate controllability.
By contradiction, we suppose that the pair $(A,B)$ did not satisfy Kalman's controllability rank condition  (\ref{th-1-1}). Then we would have that
\begin{equation}\label{lm-7-9}
  \mbox{rank}\;(B, AB, A^2B, \ldots, A^{n-1}B) < n.
\end{equation}
By (\ref{lm-7-9}), we can use  Lemma \ref{lemmawang3.3} to find
$\hat{\mathbf{z}} \in L^2(\Omega; \mathbb R^n)\setminus\{0\}$ so that
\begin{equation}\label{lm-7-11-11}
  \chi_{\omega}B^*e^{\mathcal A^*t}\hat {\mathbf{z}} = 0 \;\;\mbox{for all}\;\; t >0.
\end{equation}

 Meanwhile, according to the approximate controllability of  the system (\ref{heat-eq-ip}) (see (ii) of the Definition \ref{Def-3}), there exists $p\in\mathbb N^+$ and an increasing sequence $\{\tau_{k}\}_{k=1}^{p} \subset(0,1)$ so that for each $\varepsilon>0$, there is $\{\mathbf{v}_{k,\varepsilon}\}_{k=1}^{p} \subset L^2(\Omega; \mathbb R^m)$  so that
\begin{equation}\label{lm-7-14}
  \|\mathbf{y}(1; 0, \{\tau_{k}\}_{k=1}^{p}, \{\mathbf{v}_{k,\varepsilon}\}_{k=1}^{p}) -\hat{\mathbf{z}}\| \leq \varepsilon.
\end{equation}
By (\ref{lm-1-1}), (\ref{lm-7-11-11}), the Cauchy-Schwarz inequality and (\ref{lm-7-14}), we get that for each $\varepsilon > 0$,
\begin{eqnarray*}
\langle \hat {\mathbf{z}}, \hat {\mathbf{z}} \rangle
&=& \big\langle\mathbf{y}(1; 0, \{\tau_{k}\}_{k=1}^{p}, \{\mathbf{v}_{k,\varepsilon}\}_{k=1}^{p}),\hat {\mathbf{z}} \big\rangle
 +\big\langle\hat {\mathbf{z}} - \mathbf{y}(1; 0, \{\tau_{k}\}_{k=1}^{p}, \{\mathbf{v}_{k,\varepsilon}\}_{k=1}^{p}),\hat {\mathbf{z}} \big\rangle\\
%&=&\Big\langle \sum_{k=1}^n e^{\mathcal A(T-\tau_k)}\chi_{\omega}Bv_k,\hat {\mathbf{z}} \Big\rangle - \big\langle\hat {\mathbf{z}} - \mathbf{y}(T; \mathbf{y}_0, \{\tau_{k,\varepsilon}\}_{k=1}^{p}, \{v_{k,\varepsilon}\}_{k=1}^{p}),\hat {\mathbf{z}} \big\rangle\\
&=&\sum_{k=1}^{p} \Big\langle  \mathbf{v}_{k,\varepsilon},\chi_{\omega}B^*e^{\mathcal A^*(T-\tau_{k})}\hat{\mathbf{z}} \Big\rangle
+ \big\langle\hat {\mathbf{z}} - \mathbf{y}(1; \mathbf{y}_0, \{\tau_{k}\}_{k=1}^{p}, \{\mathbf{v}_{k,\varepsilon}\}_{k=1}^{p}),\hat {\mathbf{z}} \big\rangle\\
&=&\ \big\langle\hat {\mathbf{z}} - \mathbf{y}(1; 0, \{\tau_{k}\}_{k=1}^{p}, \{\mathbf{v}_{k,\varepsilon}\}_{k=1}^{p}),\hat {\mathbf{z}} \big\rangle\\
&\leq & \|\hat {\mathbf{z}} - \mathbf{y}(1; \mathbf{y}_0, \{\tau_{k}\}_{k=1}^{p}, \{\mathbf{v}_{k,\varepsilon}\}_{k=1}^{p})\| \|\hat {\mathbf{z}}\| \leq \varepsilon \|\hat {\mathbf{z}}\|.
\end{eqnarray*}
This implies that $\hat {\mathbf{z}} = 0$, which leads to a contradiction, since $\hat {\mathbf{z}}\neq 0$. Hence, $(A,B)$ satisfies Kalman's controllability rank condition (\ref{th-1-1}). This proves the necessity.

\vskip 5pt

(ii) The conclusion (ii)  has been proved in Step 1 of the proof of the conclusion (i).

\vskip 5pt
In summary, we end the proof of Theorem \ref{th-2}

\end{proof}

The next Example~\ref{e-2} explains the rationality of the condition that $\tau_n-\tau_1<d_A$ in (ii) in Theorem \ref{th-2}. Here,
$d_A$ is given by
(\ref{th-o-1-1}) where $C=A$.

\begin{example}\label{e-2}
Let $(A,B)\in \mathbb{R}^{2\times 2}\times \mathbb{R}^{2\times 1}$ be the pair given by (\ref{wanggssg2.3})-(\ref{e1-24}).
From Example \ref{e-1}, we see that $(A,B)$ satisfies Kalman's controllability rank condition (\ref{th-1-1}) and that $d_A={\pi}/{|b|}$.
  We will show that for any $T>0$, any $\tau_1, \tau_2\in(0,T)$,
with $\tau_2-\tau_1=d_A$,  the approximate controllability for the system (\ref{heat-eq-ip}) (governed by this pair $(A,B)$) over $[0,T]$ cannot be realized at $\{\tau_k\}_{k=1}^2$.

To this end, we arbitrarily fix $T>0$,  and then fix  $\tau_1, \tau_2\in(0,T)$,
with $\tau_2-\tau_1=d_A$.
From Example \ref{e-1}, we have that
\begin{equation}\label{p-1}
  \mbox{rank}\;(e^{A\tau_1}B, e^{A\tau_2}B) <2.
\end{equation}
Define a subspace of $V\subset \mathbb R^2$ by
\begin{equation}\label{p-2}
  V \triangleq \{e^{A\tau_1}B\alpha_1 + e^{A\tau_2}B\alpha_2\;: \;\; \alpha_1,\alpha_2 \in \mathbb R^1 \}.
\end{equation}
 By (\ref{p-1}) and (\ref{p-2}), there is  $\hat\alpha \in \mathbb R^2\setminus\{0\}$ so that
\begin{equation}\label{p-3}
  \hat \alpha^\intercal \beta = 0,\;\; \forall\; \beta \in V.
\end{equation}
Define two functions $y_0$ and $\hat{\mathbf z}_T$ over $\Omega$ in the following manner:
\begin{equation}\label{p-4}
 \mathbf y_0(x) \equiv e^{AT}\hat{\mathbf z}_T(x) \;\;\mbox{for all}\;\;x\in \Omega \;\;\mbox{and}\;\;
 \hat{\mathbf z}_T(x) \equiv e^{A^*T}\hat\alpha\;\;\mbox{for all}\;\; x\in\Omega.
\end{equation}
We claim that there exists $\delta(T) >0$ so that for any $\{\mathbf{u}_k\}_{k=1}^2\subset L^2(\Omega; \mathbb{R}^1)$,
\begin{equation}\label{p-7}
  \|\mathbf{y}(T; \mathbf{y}_0, \{\tau_k\}_{k=1}^2, \{\mathbf{u}_k\}_{k=1}^2)\| \geq \delta(T).
\end{equation}
To this end, we arbitrarily fix $\{\mathbf{u}_k\}_{k=1}^2\subset L^2(\Omega; \mathbb{R}^1)$.
By (\ref{p-4}) and the second equality in Proposition \ref{lm-5}, we get that
\begin{eqnarray}\label{p-5}
  \big<\hat{\mathbf z}_T, \sum_{k=1}^2e^{\mathcal A(T-\tau_k)}B\chi_{\omega}\mathbf{u}_k \big>
   &=& \sum_{k=1}^2\int_{\Omega}\big<e^{A^*T}\hat\alpha, e^{\mathbf{\Delta}(T-\tau_k)} e^{-A(T-\tau_k)}B\chi_{\omega}\mathbf{u}_k(x)  \big>_{\mathbb R^1}dx
\nonumber\\
   &=& \sum_{k=1}^2\int_{\Omega}\big<\hat\alpha, e^{AT}e^{\mathbf{\Delta}(T-\tau_k)} e^{-AT}e^{A\tau_k}B\chi_{\omega}u_k(x)  \big>_{\mathbb R^1}dx  .
\end{eqnarray}
Meanwhile, by (\ref{p-2}), we see that when $k\in\{1,2\}$,
\begin{eqnarray*}
e^{-AT}e^{A\tau_k}B\chi_{\omega}\mathbf{u}_k(x) \in e^{-AT}(V)\;\;\mbox{for a.e.}\;\;x\in\Omega.
\end{eqnarray*}
From the above, we can apply
 (i) of Proposition \ref{lm-6} (where
$\omega_1=\Omega, \tau=T-\tau_k$ and $\mathbf{z}=e^{-AT}e^{A\tau_k}B\chi_{\omega}\mathbf{u}_k$) to obtain that
\begin{equation}\label{wanggs5.19}
   e^{\mathbf{\Delta}(T-\tau_k)}e^{-AT}e^{A\tau_k}B\chi_{\omega}(x)\mathbf{u}_k(x) \in e^{-AT}(V)\;\;\mbox{for a.e.}\;\;x\in\Omega.
\end{equation}
Now, by (\ref{wanggs5.19}), (\ref{p-5}) and (\ref{p-3}), it follows that
\begin{equation}\label{p-5-1}
  \big<\hat{\mathbf z}_T, \sum_{k=1}^2e^{\mathcal A(T-\tau_k)}B\chi_{\omega}\mathbf{u}_k \big>=0.
\end{equation}
Then,  from (\ref{solution-formula}), (\ref{p-5-1}), (\ref{p-4}) and the second equality in  Proposition \ref{lm-5}, we see that
\begin{eqnarray}\label{p8}
&&\big<\hat{\mathbf z}_T, \mathbf{y}(T; \mathbf{y}_0, \{\tau_k\}_{k=1}^2, \{\mathbf{u}_k\}_{k=1}^2) \big>
= \big<\hat{\mathbf z}_T, e^{\mathcal AT}\mathbf{y}_0 \big>
+\big<\hat{\mathbf z}_T, \sum_{k=1}^2e^{\mathcal A(T-\tau_k)}B\chi_{\omega}\mathbf{u}_k \big>
\nonumber\\
&=& \int_{\Omega}\big<\hat{\mathbf z}_T(x), e^{\mathbf{\Delta}T}e^{-AT}\mathbf{y}_0(x) \big>_{\mathbb R^1}dx
= \int_{\Omega}\big<\hat{\mathbf z}_T(x), e^{\mathbf{\Delta}T}\hat{\mathbf z}_T(x)\big>_{\mathbb R^1}dx
\nonumber\\
&=& \|e^{\mathbf{\Delta}\frac{T}{2}}\hat{\mathbf z}_T\|^2.
\end{eqnarray}
On the other hand, it follows by the Cauchy-Schwarz inequality that
\begin{eqnarray*}
  \big<\hat{\mathbf z}_T, \mathbf{y}(T; \mathbf{y}_0, \{\tau_k\}_{k=1}^2, \{\mathbf{u}_k\}_{k=1}^2) \big>
  \leq \|\hat{\mathbf z}_T\|\|\mathbf{y}(T; \mathbf{y}_0, \{\tau_k\}_{k=1}^2, \{\mathbf{u}_k\}_{k=1}^2)\|.
\end{eqnarray*}
This, along with (\ref{p8}), yields that
\begin{eqnarray*}
 \|\mathbf{y}(T; \mathbf{y}_0, \{\tau_k\}_{k=1}^2, \{\mathbf{u}_k\}_{k=1}^2)\|
 \geq \frac{\|e^{\mathbf{\Delta}\frac{T}{2}}\hat{\mathbf z}_T\|^2}{\|\hat{\mathbf z}_T\|},
\end{eqnarray*}
which leads to (\ref{p-7}), since $\hat{\mathbf z}_T\neq 0$ only depends on $T$ (see (\ref{p-4})).

Finally, from  (\ref{p-7}), we see that the approximate controllability for the system (\ref{heat-eq-ip}) (governed by this pair $(A,B)$) over $[0,T]$ cannot be realized at any $\{\tau_k\}_{k=1}^2\subset (0,T)$ with $\tau_2-\tau_1=d_A$.

Besides, this example also shows that for each $T>0$, the approximate controllability for the system (\ref{heat-eq-ip}) (governed by this pair $(A,B)$) over $[0,T]$ cannot be realized at a single control instant $\tau\in (0,T)$. Let us explain the reason . Let $T>0$. Since  $\mbox{rank}\;(e^{A\tau}B) <2$ for all $\tau\in (0,T)$ (see Example \ref{e-1}), it follows from Theorem~\ref{lm-4} that there is $\hat {\mathbf{z}}\in L^2(\Omega; \mathbb{R}^2)\setminus\{0\}$ so that
\begin{equation}\label{wang5.22}
  \chi_{\omega}B^*e^{\mathcal A^*\tau} \hat {\mathbf{z}}= 0 \;\; \mbox{for all}\;\; \tau\in(0,T).
\end{equation}
On the other hand, one can easily check that the approximate controllability for the system (\ref{heat-eq-ip}) over $[0,T]$ can be realized at a single control instant $\tau\in (0,T)$
if and only if
\begin{equation*}
 \mathbf{z}\in L^2(\Omega;\mathbb{R}^n)\;\;\mbox{and}\;\; \chi_{\omega}B^*e^{\mathcal A^*(T-\tau)}  {\mathbf{z}}= 0\Longrightarrow \mathbf{z}=0.
\end{equation*}
This, along with (\ref{wang5.22}), yields that   the approximate controllability for the system (\ref{heat-eq-ip}) (governed by this pair $(A,B)$) over $[0,T]$ cannot be realized at a single control instant $\tau\in (0,T)$.

\end{example}

\end{document}